\documentclass[10pt]{amsart}
\usepackage[margin=3cm]{geometry}
\usepackage{hyperref}
\usepackage{pb-diagram}
\usepackage{amsmath, amssymb}
\usepackage{tikz}
\usetikzlibrary{matrix,arrows,decorations.pathmorphing}
\usepackage{paralist}
\usepackage{xcolor}
\usepackage{mathrsfs}

\newtheorem{thm}{Theorem}
\newtheorem{lem}[thm]{Lemma}
\newtheorem{prp}[thm]{Proposition}

\theoremstyle{definition}
\newtheorem{df}[thm]{Definition}
\newtheorem{exa}[thm]{Example}
\newtheorem{cor}[thm]{Corollary}
\newtheorem{rem}[thm]{Remark}

\theoremstyle{remark}

\numberwithin{equation}{section}
\numberwithin{thm}{section}

\newcommand{\obsol}[1]{}

\DeclareMathOperator{\im}{im}

\DeclareMathOperator{\supp}{supp}
\DeclareMathOperator{\st}{\;|\;}
\DeclareMathOperator{\map}{map}

\DeclareMathOperator{\hocolim}{hocolim}
\DeclareMathOperator{\colim}{colim}
\DeclareMathOperator*{\Lhocolim}{hocolim}

\DeclareMathOperator{\Ob}{Ob}

\DeclareMathOperator{\sd}{sd}
\DeclareMathOperator{\Ch}{Ch}

\DeclareMathOperator{\dom}{dom}

\DeclareMathOperator{\alt}{alt}
\DeclareMathOperator{\carr}{carr}
\DeclareMathOperator{\const}{const}
\DeclareMathOperator{\Id}{Id}

\def\Top{\mathbf{Top}}
\def\hTop{\mathbf{hTop}}
\def\dTop{\mathbf{dTop}}
\def\Set{\mathbf{Set}}

\def\R{\mathbb{R}}
\def\Z{\mathbb{Z}}

\def\bO{\mathbf{0}}
\def\bI{\mathbf{1}}
\def\ba{\mathbf{a}}
\def\bb{\mathbf{b}}
\def\bc{\mathbf{c}}
\def\be{\mathbf{e}}

\def\bm{\mathbf{m}}
\def\bn{\mathbf{n}}

\def\bx{\mathbf{x}}

\def\cC{\mathcal{C}}

\def\cE{\mathcal{E}}

\def\scU{\mathscr{U}}

\def\vN{\vec{N}}
\def\vP{\vec{P}}

\def\vI{\vec{I}}

\DeclareMathOperator\OConf{OConf}
\DeclareMathOperator\UConf{UConf}

\newcommand\georel[1]{\left| #1 \right|}

\DeclareFontFamily{U}{min}{}
\DeclareFontShape{U}{min}{m}{n}{<-> udmj30}{}

\makeatletter
\newcommand{\oset}[3][0ex]{%
  \mathrel{\mathop{#3}\limits^{
    \vbox to#1{\kern-2\ex@
    \hbox{$\scriptstyle#2$}\vss}}}}
\makeatother
\def\dle{\mathrel{\oset{*}{<}}}
\def\xle{\mathrel{\oset{x}{<}}}
\def\yle{\mathrel{\oset{y}{<}}}

\def\sqsubsetneq{\mathrel{\sqsubseteq\kern-0.92em\raise-0.15em\hbox{\rotatebox{313}{\scalebox{1.1}[0.75]{\(\shortmid\)}}}\scalebox{0.3}[1]{\ }}}

\begin{document}

\title{Configuration spaces and directed paths on the final precubical set}

\author[J. Paliga]{Jakub Paliga}
\address{
Faculty of Mathematics, Informatics and Mechanics\\ University of Warsaw\\
Banacha 2\\
02-097 Warszawa, Poland}
\email{jp371350@students.mimuw.edu.pl}

\author[K. Ziemia\'nski]{Krzysztof Ziemia\'nski}
\address{
Faculty of Mathematics, Informatics and Mechanics\\ University of Warsaw\\
Banacha 2\\
02-097 Warszawa, Poland}
\email{ziemians@mimuw.edu.pl}

\begin{abstract}
The main goal of this paper is to prove that the space of directed loops on the final precubical set is homotopy equivalent to the ``total" configuration space of  points on the plane; by ``total" we mean that any finite number of points in a configuration is allowed.
We also provide several applications: we define new invariants of precubical sets, prove that directed path spaces on any precubical complex have the homotopy types of CW-complexes and construct certain presentations of configuration spaces of points on the plane as nerves of categories.
\end{abstract}



\maketitle

\section{Introduction}

Directed Algebraic Topology (DAT) is a branch of mathematics that uses topological methods to solve problems arising in computer science, especially in concurrency \cite{FGHMR}. The main objects of interest of DAT are directed spaces (d-spaces) \cite{Grandis}. A directed space is a topological space $X$ with a distinguished family of paths $\vP(X)$, called directed paths or d-paths. Points of $X$ represent possible states of a concurrent program, while directed paths indicate how the states can evolve in time. The space of directed paths between two distinguished points of $X$, the initial point $\bO$ and the final point $\bI$, is the space of executions of a bi-pointed directed space $(X,\bO,\bI)$.

The main source of interesting examples of directed spaces are geometric realizations of precubical sets, also written $\square$--sets for short.
These objects are used in computer science for modeling concurrent automata;
for example Higher Dimensional Automata, introduced by Pratt \cite{Pratt}, are $\square$--sets enriched with some additional structure.
A $\square$--set $K$ is a sequence of sets $(K[n])_{n\geq 0}$ of $n$--cubes together with face maps, which determine how these cubes are glued.
We say that $K$ is bi-pointed if it is equipped with the initial and the final vertex $\bO,\bI\in K[0]$.
For a (bi-pointed) $\square$--set $K$ we associate its geometric realization $|K|$, which is a (bi-pointed) directed space.
The space $\vP(|K|)_\bO^\bI$ of directed paths from $\bO$ to $\bI$ will be called the execution space of $K$.
For short, we will usually skip the vertical bars and write $\vP(K)_\bO^\bI$ for $\vP(|K|)_\bO^\bI$.

Let $Z$ be a $\square$--set with exactly one cube in every dimension.
It is a final object both in the category of $\square$--sets and in the category of bi-pointed $\square$--sets
(the only vertex of $Z$ is then regarded as both the initial and the final vertex).
One could expect that the final object does not carry any nontrivial structure,
but this is not the case for $\square$--sets.
Let $\UConf(n,\R^2)$ denote the space of unordered configurations of $n$ points on the real plane.
The main result of this paper is
\begin{thm}\label{t:Main}
	There is a homotopy equivalence
	\[
		\vP(Z)_\bO^\bI \simeq \coprod_{n\geq 0} \UConf(n,\R^2).
	\]
\end{thm}

\subsection*{Guide to the argument}
Following  \cite{Raussen-Trace}, we introduce the length of a directed path on (the geometric realization of) a bi-pointed $\square$--set $K$.
The length of a d-path that starts and ends at a vertex is a non-negative integer,
and homotopic d-paths have equal lengths.
As a consequence, $\vP(K)_\bO^\bI$ decomposes into the disjoint union of the spaces $\vP(K;n)_\bO^\bI$ of d-paths having length $n$.
Then we construct, for every $n\geq 0$, a bi-pointed $\square$--set $\tilde{K}_n$ equipped with a $\square$--map $\tilde{K}_n\to K$
that induces a homeomorphism $\vP(K;n)_\bO^\bI\cong \vP(\tilde{K}_n)_\bO^\bI$.
To prove the main theorem,
we need to show that there is a homotopy equivalence $\vP(\tilde{Z}_n)_\bO^\bI\simeq \UConf(n,\R^2)$.

A method of calculating the homotopy type of the execution space of $K$ is described in \cite{Z-Cub2}. 
Namely, $\vP(K)_\bO^\bI$ is weakly homotopy equivalent to the nerve of the cube chain category $\Ch(K)$ of $K$.
If $K$ is non-self-linked, then $\vP(K)_\bO^\bI$ has the homotopy type of a CW-complex, 
so we obtain a genuine homotopy equivalence.
Unfortunately, $\tilde{Z}_n$ is not non-self-linked for $n>1$,
so we cannot apply the methods of \cite{Z-Cub2} directly.
Instead, for a fixed set $A$ having $n$ elements,
we construct a non-self-linked $\square$--set $Y^A$
with an action of the group $\Sigma_A$ of permutations of $A$,
and a $\square$--map $Y^A\to\tilde{Z}_n$.
It turns out that the induced $\Sigma_A$--action on $\vP(Y^A)_\bO^\bI$ is free,
and the induced map between execution spaces  
$\vP(Y^A)_\bO^\bI\to \vP(\tilde{Z}_n)_\bO^\bI$
is a $\Sigma_A$--principal bundle,
ie, $\vP(\tilde{Z}_n)_\bO^\bI\simeq \vP(Y^A)_\bO^\bI/\Sigma_A$.
Now we can use the cube chain decomposition
and prove that $\vP(Y^A)_\bO^\bI$ is $\Sigma_A$--homotopy equivalent 
to $|\Ch(Y^A)|$.

Similarly, the unordered configuration space $\UConf(n;\R^2)$ is the quotient space of the ordered configuration space 
\begin{equation}
	\OConf(A,\R^2)=\{f:A\to \R^2\st \text{$f$ is injective}\},
\end{equation}
where $A$ is a set having $n$ elements, by the (free) action of $\Sigma_A$.
Thus, we need to prove that the spaces $|\Ch(Y^A)|$ and $\OConf(A;\R^2)$ are $\Sigma_A$--homotopy equivalent.

To compare these spaces, we introduce double orders.
A double order on a set $A$ is a pair of (strict) partial orders on $A$ that satisfies certain conditions (Definition \ref{d:DoublyOrderedSet}).
We introduce two different (non-strict) partial orders, denoted by ${\subseteq}$ and ${\sqsubseteq}$ respectively, on the set of double orders $D(A)$.

It turns out that the category $\Ch(Y^A)$ is equivariantly isomorphic with the poset $(R(A),\sqsupseteq)$,
where $R(A)\subseteq D(A)$ is the subset of regular double orders.

Further on, every double order on $A$ defines an open subset of $\OConf(A,\R^2)$.
Thus, we obtain a cover of $\OConf(A,\R^2)$, indexed by the set of semi-regular double orders $R^+(A)\subseteq D(A)$,
which is $\Sigma_A$--equivariant and complete (ie, closed with respect to intersection).
As a consequence, an equivariant version of Nerve Lemma (Lemma \ref{l:NerveComplete})
gives an equivariant homotopy equivalence $\OConf(A,\R^2)\simeq |(R^+(A),\supseteq)|$.

Finally, we prove that the spaces $|(R^+(A),\supseteq)|$ and $|(R(A),\sqsupseteq)|$ are $\Sigma_A$ homotopy equivalent.

\begin{proof}[Proof of {\ref{t:Main}}]
	In the diagram
\begin{equation}\label{e:LargeDiag}
	\begin{diagram}
		\node{\vP(Y^A)_\bO^\bI}
			\arrow{s,lr}{\text{Prop.\ref{p:ChPYEqHE}}}{\simeq_{\Sigma_A}}
			\arrow{e,t}{\text{Prop.\ref{p:MainCov}}}
		\node{\vP(\tilde{Z}_n)_\bO^\bI}
	\\
		\node{|\Ch(Y^A)|}
			\arrow{s,lr}{\text{Prop.\ref{p:ChRComp}}}{\simeq_{\Sigma_A}}
	\\
		\node{|(R(A),\sqsupseteq)|}
			\arrow{s,lr}{\text{Prop.\ref{p:RRComp}}}{\simeq_{\Sigma_A}}
	\\
		\node{|(R^+(A),\supseteq)|}
			\arrow{s,lr}{\text{Prop.\ref{p:RConfEq}}}{\simeq_{\Sigma_A}}
	\\
		\node{\OConf(A,\R^2)}
			\arrow{e}
		\node{\UConf(n,\R^2)}
	\end{diagram}
\end{equation}
	both the horizontal maps are principal $\Sigma_A$--bundles
	and all the vertical maps are $\Sigma_A$--homotopy equivalences.
	Therefore,
	\[
		\vP(\tilde{Z}_n)_\bO^\bI
		\simeq 
		\vP(Y^A)_\bO^\bI/\Sigma_A
		\simeq 
		\OConf(A,\R^2)/\Sigma_A
		=
		\UConf(n,\R^2),
	\]
	and, by Proposition \ref{p:LenDecTilde},
	\[
		\vP(Z)_\bO^\bI\cong \coprod_{n\geq 0} \vP(\tilde{Z}_n)_\bO^\bI,
	\]
	which ends the proof.
\end{proof}

In the final section, we discuss some applications of results obtained in this paper.

\begin{enumerate}
\item
Since $Z$ is a final object in the category of bi-pointed $\square$--sets,
every $\square$--set $K$ is equipped with a unique $\square$--map $K\to Z$,
which in turn induces a map $\vP(K)_\bO^\bI\to\vP(Z)_\bO^\bI$.
This allows to construct some invariants of $\square$--sets.
\item
We prove that every component of the execution space of $K$ can be realized, up to finite covering,
as a component of the execution space of a non-self-linked $\square$--set.
Further on, if some invariants mentioned above vanish, we can obtain a strict realization.
\item
We show that the execution spaces of all finite $\square$--sets have the homotopy types of CW-complexes.
\item
We note that the geometric realization $|Z|$ is homeomorphic, as a (non-directed) topological space, to the James construction \cite{James} on $S^1$.
This allows to interpret configuration spaces as ``directed" self-maps of the sphere $S^2$.
\item
Finally, we present a small presentation of the unordered configuration spaces as nerves of certain categories.
\end{enumerate}

\section{$\square$--sets}

\subsection*{$\square$--sets}

\emph{A $\square$--set} (or \emph{a precubical set}) $K$ is a sequence of pairwise disjoint sets $(K[n])_{n\geq 0}$ equipped with the \emph{face maps} $d^\varepsilon_i:K[n]\to K[n-1]$ for $i\in \{1,\dotsc,n\}$, $\varepsilon\in\{0,1\}$ that satisfy the precubical relations: $d^\varepsilon_i d^\eta_j=d^\eta_{j-1}d^\varepsilon_i$ for $i<j$.
\emph{A $\square$--map} $f:K\to L$ between $\square$--sets $K,L$ is a sequence of functions $f[n]:K[n]\to L[n]$ that commute with the face maps.

Elements $K[n]$ will be called \emph{$n$--cubes}; $0$--cubes and $1$--cubes will be called \emph{vertices} and \emph{edges}, respectively.

\emph{A bi-pointed} $\square$--set is a triple $(K,\bO_K,\bI_K)$, where $K$ is a $\square$--set and $\bO_K,\bI_K\in K[0]$ are \emph{the initial} and \emph{the final vertex}, respectively.
A precubical map $f:K\to L$ between bi-pointed $\square$--sets is \emph{bi-pointed} if $f(\bO_K)=\bO_L$ and $f(\bI_K)=\bI_L$.
We will skip the lower indices at $\bO$ and $\bI$ whenever it does not lead to confusion.

$\square$--sets and $\square$--maps form a category that is denoted $\square\Set$. Similarly, the category of bi-pointed $\square$--sets and bi-pointed $\square$--maps will be denoted by $\square\Set_*^*$.

\begin{exa}
	Let $A$ be a finite totally ordered set. \emph{The standard $A$--cube} $\square^A$ is a precubical set such that
	\begin{itemize}
	\item
		$\square^A[m]$ is the set of functions $c:A\to \{0,1,*\}$ such that $c^{-1}(*)$ has exactly $m$ elements,
	\item
		If $c^{-1}(*)=\{a_1<\dotsm<a_m\}$, then
\[
		d_i^\varepsilon(c)(a)=\begin{cases}
			\varepsilon & \text{for $a=a_i$}\\
			c(a) & \text{for $a\neq a_i$.}
		\end{cases}
\]		
	\end{itemize}
	When $\square^A$ is regarded as a bi-pointed $\square$--set, we take the constant functions $\bO_{\square^A}(a)=0$, $\bI_{\square^A}(a)=1$ as the initial and final vertices, respectively.
\end{exa}

	Denote $\square^n=\square^{\{1<\dotsm<n\}}$. Let $u_n$ be the only $n$--cube of $\square^n$; we have $u_n(i)=*$ for all $i$. For any $\square$--set $K$ and any $c\in K[n]$, there is a unique $\square$--map $\iota_c:\square^n\to K$ such that $\iota_c(u_n)=c$, which is called \emph{the canonical map} related to the $n$--cube $c$.

\begin{exa}\label{x:Z}
	Let $Z$ be a $\square$--set such that $Z[m]=\{z^m\}$ for all $m\geq 0$; 
	the face maps are determined uniquely.
	$Z$ is a final object in $\square\Set$ and $(Z,z^0,z^0)$ is a final object in $\square\Set_*^*$ as well.
\end{exa}

\subsection*{Faces}

Let $K$ be a $\square$--set.
A cube $b\in K[n-r]$ is \emph{a face} of $c\in K[n]$ if there exist $i_1,\dotsc,i_r$ and $\varepsilon_1,\dotsc,\varepsilon_r$ such that
	\begin{equation}\label{e:FacePresentation}
		b=d^{\varepsilon_1}_{i_1} d^{\varepsilon_2}_{i_2}\dotsc d^{\varepsilon_r}_{i_r}(c)	
	\end{equation}
This presentation can be always chosen so that $i_1<\dotsm <i_r$ (\cite[Lemma 6.11]{FGR}).

For a subset $A=\{a_1<\dotsm<a_k\}\subseteq \{1,\dotsc,n\}$, $\varepsilon\in\{0,1\}$ and $c\in K[n]$, denote
\begin{equation}\label{e:DDef}
	d^\varepsilon_A(c)=d^\varepsilon_{a_1} d^{\varepsilon}_{a_2}\dotsm d^\varepsilon_{a_k}(c).
\end{equation}
Vertices $d^0(c)=d^0_{\{1,\dotsc,n\}}(c)$ and $d^1(c)=d^1_{\{1,\dotsc,n\}}(c)$ will be called \emph{the initial} and \emph{the final vertex} of $c$, respectively.

\subsection*{Directed spaces}

\emph{A d-space} \cite{Grandis}, or \emph{a directed space},  is a topological space $X$ with a family $\vP(X)\subseteq P(X)=\map([0,1],X)$ of paths, called \emph{d-paths} (or \emph{directed paths}), that contains all constant paths and is closed with respect to concatenations and increasing reparametrizations. The concatenation of paths $\alpha,\beta\in P(X)$, $\alpha(1)=\beta(0)$, is the path
\[
	\alpha*\beta(t)=
	\begin{cases}
		\alpha(2t) & \text{for $t\in[0,\tfrac{1}{2}]$,}\\
		\beta(2t-1) & \text{for $t\in[\tfrac{1}{2},1]$.}
	\end{cases}
\]
A continuous map $f:X\to Y$ between d-spaces is \emph{a d-map} if it preserves d-paths;
d-spaces and d-maps form a category $\dTop$, which is complete and cocomplete.
\emph{A bi-pointed d-space} is a d-space $X$ with two chosen points $\bO_X,\bI_X\in X$.

For $x,y\in X$ denote
\[
	\vP(X)_x^y=\{\alpha\in \vP(X)\;|\;\text{$\alpha(0)=x$ and $\alpha(1)=y$}\}.
\]
For $s<t$, a path $[s,t]\to X$ will be called a d-path if its linear reparametrization to the unit interval is a d-path. The space of such paths will be denoted $\vP_{[s,t]}(X)$, or $\vP_{[s,t]}(X)_x^y$ if the endpoints are fixed.

Examples of d-spaces are:
\begin{itemize}
\item
	\emph{The directed interval} $\vI$, which is the unit interval $I=[0,1]$ with increasing continuous maps $I\to I$ as d-paths.
\item
	\emph{The directed $n$--cube $\vI^n$}, which is the $n$--fold Cartesian product of directed intervals $\vI$. A path $\alpha=(\alpha_1,\dotsc,\alpha_n)\in P(\vI^n)$ is a d-path if all its coordinates $\alpha_i$ are increasing.
\end{itemize}

\subsection*{Geometric realization}

For $n>0$, $\varepsilon\in\{0,1\}$ and $i\in\{1,\dotsc,n\}$ define a map
\begin{equation}\label{e:DeltaMaps}
	\delta^\varepsilon_i:\vI^{n-1} \ni (x_1,\dotsc,x_{n-1})\mapsto (x_1,\dotsc,x_{i-1},\varepsilon,x_i,\dotsc,x_{n-1})\in \vI^n.
\end{equation}

\emph{The geometric realization} of a $\square$--set $K$ is the d-space
	\[
		\georel{K}=\coprod_{n\geq 0} K[n]\times \vI^n/\sim,
	\]
	where the relation $\sim$ is generated by $(d_i^\varepsilon(c), \bx)\sim (c, \delta^\varepsilon_i(\bx))$ for all $c\in K[n]$, $\bx\in \vI^{n-1}$, $i\in\{1,\dotsc,n\}$, $\varepsilon\in\{0,1\}$.
The equivalence class of $(c,\bx)$, $c\in K[n]$, $\bx\in\vI^n$ will be denoted by $[c;\bx]$. 	
	As in (\ref{e:DDef}), we define $\delta^\varepsilon_A=\delta^\varepsilon_{a_k} \delta^{\varepsilon}_{a_{k-1}}\dotsm \delta^\varepsilon_{a_1}$
	for $A=\{a_1<\dotsm<a_k\}$; obviously $[d_A^\varepsilon(c), \bx]=[c, \delta^\varepsilon_A(\bx))$ for $c\in K[n]$, $A\subseteq \{1,\dotsc,n\}$ and  $\bx\in \vI^{n-|A|}$.

The geometric realization defines a functor $\georel{-}:\square\Set\to \dTop$ and its bi-pointed counterpart $\square\Set_*^*\to\dTop_*^*$.

Every point $p\in |K|$ has the unique \emph{canonical presentation}
\begin{equation}
	[c;\bx]=[c;(x_1,\dotsc,x_n)]
\end{equation}
such that $x_i\neq 0,1$ for all $i$.
The cube $c$ is called \emph{the carrier} of $p$ and denoted $\carr(p)$.

The following is shown in \cite[Section 3]{Z-Cub2}.
\begin{prp} \label{p:DPathPres}
Every d-path $\alpha\in\vP(K)$ has a presentation
\begin{equation}\label{e:PathPresentation}
	\alpha = [c_1;\beta_1]\overset{t_1}{*}[c_2;\beta_2]\overset{t_2}{*}\dotsm\overset{t_{m-1}}{*}[c_m;\beta_m]
\end{equation}
for $0=t_0<t_1<\dotsm<t_{m-1}<t_m=1$, $c_i\in K[n_i]$, $\beta_i\in \vP_{[t_{m-1},t_m]}(\vI^{n_i})$.
Furthermore, we may assume that:
\begin{enumerate}[\normalfont(a)]
\item
	 no coordinate of $\beta_i(t_{i-1})$ is equal to $1$,
\item
	no coordinate of $\beta_i(t_i)$ is equal to $0$,
\item
	$d^1_{B_i}(c_i)=d^0_{A_{i+1}}(c_{i+1})$ for some sets $A_i,B_i\subseteq \{1,\dotsc,n\}$,
\item
	there exist $\bx_i\in \vI^{n_i-|B_i|}=\vI^{n_{i+1}-|A_{i+1}|}$ such that $\beta_i(t_i)=\delta^1_{B_i}(\bx_i)$, $\beta_{i+1}(t_i)=\delta^0_{A_{i+1}}(\bx_i)$.\qed
\end{enumerate}
\end{prp}

\begin{exa}
	The geometric realization of the standard $n$--cube $\square^n$ is the directed $n$--cube $\vI^n$.
\end{exa}

\def\len{\mathrm{len}}

\subsection*{Cube chains}
Here we follow \cite{Z-Cub2}.
Let $K$ be a bipointed precubical set. \emph{A cube chain} on $K$ is a sequence of cubes $\bc=(c_1,\dotsc,c_l)$, $c_i\in K[n_i]$, $n_i>0$ such that
	\begin{itemize}
	\item
		$d^0(c_1)=\bO_K$,
	\item
		$d^1(c_l)=\bI_K$,
	\item
		$d^1(c_i)=d^0(c_{i+1})$ for $i=1,\dotsc,l-1$.
	\end{itemize}
\emph{The dimension} of a cube chain $\bc$ is the sequence $\bn^\bc=(n_1,\dotsc,n_l)$ and \emph{the length} of $\bc$ is $\len(\bc)=n_1+\dotsm+n_l$.
Let $\Ch(K)$ and $\Ch(K;n)$ denote the set of cube chains on $K$, and cube chain on $K$ having length $n$, respectively.

We need to define morphisms between cube chains to make $\Ch(K)$ into a category.
\emph{The serial wedge} of bi-pointed $\square$--sets $K$, $L$ is a bi-pointed $\square$--set
\[
	K\vee L = (K\sqcup L)/ \bI_K \sim \bO_L
\]
with $\bO_{K\vee L}=\bO_K$, $\bI_{K\vee L}=\bI_L$.
For a sequence $\bn=(n_1,\dotsc,n_{l})$ of positive integers, \emph{the $\bn$--wedge cube} is a $\square$--set
\[
	\square^{{\vee}\bn}=\square^{n_1}\vee\dotsm\vee \square^{n_l}.
\]
There is 1-1 correspondence between cube chains $\bc=(c_i)_{i=1}^l$ such that $c_i\in K[n_i]$ and bi-pointed $\square$--maps $\square^{{\vee}\bn}\to K$:
the map corresponding to $\bc$ is determined by the condition $\bc(u_{n_i})=c_i$.
\emph{The cube chain category} of $K$ is the category with $\square$--maps $\bc:\square^{{\vee}\bn}\to K$ as objects, and commutative diagrams
\[
	\begin{diagram}
		\node{\square^{{\vee}\bm}}
			\arrow{e,t}{\ba}
			\arrow{s}
		\node{K}
			\arrow{s,r}{=}
	\\
		\node{\square^{{\vee}\bn}}
			\arrow{e,t}{\bb}
		\node{K}
	\end{diagram}
\]
as morphisms from $\ba$ to $\bb$. This category will be also denoted by $\Ch(K)$.

The importance of the cube chain category comes from the following.
\begin{thm}[{\cite[Theorem 7.6]{Z-Cub2}}]\label{thm:CubeChainModel}
	For every bi-pointed $\square$--set $K$, the spaces $\vP(K)_\bO^\bI$ and $\georel{\Ch(K)}$ are naturally weakly homotopy equivalent (ie, as functors $\square\Set_*^*\to\hTop$).
\end{thm}

It is known that $\vP(K)_\bO^\bI$ has a homotopy type of a CW-complex 
if $K$ is non-self-linked \cite{Raussen-Trace}.
We extend this result to arbitrary finite $\square$--sets (Proposition \ref{p:PKisCW}).
Therefore, the spaces in Theorem \ref{thm:CubeChainModel} are homotopy equivalent (Corollary \ref{c:ChHE}).

\subsection*{Length of d-paths}
Following \cite{Raussen-Trace},
we define the length of a d-path $\alpha=(\alpha_1,\dotsc,\alpha_n)\in \vP(\vI^n)$ as $\mathrm{len}(\alpha)=\sum_{i=1}^n (\alpha_i(1)-\alpha_i(0))$. 
This definition extends to d-paths on geometric realizations of arbitrary $\square$--sets obeying the following conditions.
\begin{itemize}
\item
	The geometric realizations of $\square$--maps preserve length, ie, $\len({\georel{f}}\circ \alpha)=\len(\alpha)$ for every $\square$--map $f:K\to L$ and $\alpha\in\vP(K)$,
\item
	Length is additive, ie, $\len(\alpha*\beta)=\len(\alpha)+\len(\beta)$  for $\alpha,\beta\in\vP(K)$, $\alpha(1)=\beta(0)$.
\end{itemize}

The length of every bi-pointed d-path is a non-negative integer,
and homotopic d-paths have the same length.
Thus, for every bi-pointed $\square$--set $K$ there is a presentation
\begin{equation}\label{e:LengthDecomp}
	\vP(K)_\bO^\bI \cong \coprod_{n\geq 0} \vP(K; n)_\bO^\bI,
\end{equation}
where $\vP(K; n)_\bO^\bI$ stands for the space of d-paths having length $n$.

\subsection*{Altitude function}
\emph{An altitude function} on a $\square$--set $K$ is a function $\alt:\coprod_{n\geq 0} K[n]\to \Z$
such that $\alt(d^\varepsilon_i(c))=\alt(c)+\varepsilon$. If $K$ is bi-pointed, we additionally assume that $\alt(\bO_K)=0$.
If a bi-pointed $\square$--set $K$ is connected (ie, cannot be presented as a sum of disjoint sub--$\square$--sets), then there exists at most one altitude function on $K$.

\begin{exa}{\ }\label{exa:Altitude}
\begin{enumerate}[\normalfont(a)]
\item
	The standard $A$--cube $\square^A$ admits an altitude function $\alt(c)=|c^{-1}(1)|$.
\item
	If $K,L\in \square\Set_*^*$ admit altitude functions $\alt_K$ and $\alt_L$, respectively, then
	\[
		\alt(c)=
		\begin{cases}
			\alt_K(c) & \text{for $c\in K$},\\
			\alt_L(c)+\alt_K(\bI_K) & \text{for $c\in L$}
		\end{cases}
	\]
	is an altitude function on $K\vee L$.
\item
	If $\bn=(n_1,\dotsc,n_l)$ is a sequence of positive integers, then the unique altitude function on $\square^{{\vee}\bn}$ is given by
	\[
		\alt(c)=n_1+\dots+n_{i-1}+|c^{-1}(1)|
	\]
	for $c\in \square^{n_i}$.
\end{enumerate}
\end{exa}

The proof of the following is elementary.
\begin{prp}\label{p:AltitudeIsPreserved}
	Let $K$, $L$ be connected bi-pointed $\square$--sets with altitude functions $\alt_K$ and $\alt_L$, respectively. Then every bi-pointed $\square$--map $f:K\to L$ preserves altitude functions (ie, $\alt_L(f(c))=\alt_K(c)$ for all $c\in K$).\qed
\end{prp}

For an altitude function $\alt$ on $K$, we define a continuous function $\georel{\alt}:\georel{K}\to\R$ (which will be also called an altitude function) by
\[
	\georel{\alt}([c;x_1,\dotsc,x_n])=\alt(c)+x_1+\dotsc+x_n.
\]
For every d-path $\alpha$ on $\georel{K}$, we have $\len(\alpha)=\georel{\alt}(\alpha(1))-\alt(\alpha(0))$.

\begin{df}
	Let $K$ be a bi-pointed $\square$--set with an altitude function $\alt$ and denote $n=\alt(\bI)$.
	We say that a d-path $\alpha:[0,n]\to |K|$ is \emph{natural} if $\georel{\alt}(\alpha(t))=t$ for all $t\in [0,n]$.
	Let $\vN_{[0,n]}(K)_\bO^\bI$ denote the space of natural d-paths on $|K|$ from $\bO$ to $\bI$. 
\end{df}

\begin{prp}\label{p:NatProd}
	Let $K$ be a bi-pointed $\square$--set with an altitude function.
	The composition map
	\[
		\vN_{[0,n]}(K)_\bO^\bI \times \vP([0,n])_0^n \ni (\alpha,\beta) \mapsto \alpha\circ\beta\in \vP(K)_\bO^\bI
	\]
	is a functorial homeomorphism.
\end{prp}
\begin{proof}
	This is proven in \cite{Raussen-Trace}.
	There is a continuous map $\mathrm{nat}:\vP(K)_\bO^\bI\to \vN_{[0,n]}(K)_\bO^\bI$ such that $\mathrm{nat}(\alpha)(\alt(\alpha(t)))=\alpha(t)$.
	Hence, the map $\alpha\mapsto(\mathrm{nat}(\alpha), \alt\circ\alpha)$ is the inverse.
\end{proof}

\subsection*{Non-self-linked $\square$--sets}

A $\square$--set $K$ is \emph{non-self-linked} if every $\square$--map $\square^n\to K$ is injective for all $n\geq 0$.
This condition is equivalent to \cite[Def. 6.9]{FGR}.

We will need the following elementary criterion for non-self-linkedness.
\begin{prp}\label{p:NSLCov}
	Let $K$, $L_j$, $j\in J$ be $\square$--sets and let $f_j:L_j\to K$ be injective $\square$--maps. Assume that $L_j$ is non-self-linked for every $j\in J$, and that $K=\bigcup_{j\in J} f_j(L_j)$. Then $K$ is non-self-linked.\qed
\end{prp}

\begin{lem}\label{l:ChNSL}
	Let $K$ be a bi-pointed $\square$--set.
	\begin{enumerate}[\normalfont (a)]
	\item
		If $K$ is non-self-linked, then the category $\Ch(K)$ is a poset.
	\item
		If $K$ admits an altitude function, then every cube chain $\bc\in\Ch(K)$ is determined uniquely by the set of its cubes $\{c_1,\dotsc,c_{l}\}$.
	\item
		If $K$ is non-self-linked and admits an altitude function, then every bi-pointed $\square$--map $\square^{{\vee}\bn}\to K$ is injective.
		Moreover, for $\ba=(a_1,\dotsc,a_l), \bb=(b_1,\dotsc,b_m)\in\Ch(K)$, there exists a morphism from $\ba$ to $\bb$ in $\Ch(K)$ if and only if every $a_i$ is a face of some $b_j$.
	\end{enumerate}
\end{lem}
\begin{proof}
	(a)
	The category $\Ch(K)$ is an upwards directed Reedy category \cite[Section 10]{Z-Cub2}, so it is enough to show that there is at most one morphism between any two objects.
	Let $\ba:\square^{{\vee}\bm}\to K$ and $\bb:\square^{{\vee}\bn}\to K$ be cube chains in $K$. Assume that $f,g:\square^{{\vee}\bm}\to\square^{{\vee}\bn}$ are $\square$--maps such that $\ba=\bb\circ f=\bb \circ g$ (ie, $f,g\in \Ch(K)(\ba,\bb)$). 
	We will show that $f(u_{m_i})=g(u_{m_i})$ for every $i\in\{1,\dotsc,l(\bm)\}$, where $u_{m_i}$ is the top cube in $\square^{m_i}\subseteq \square^{{\vee}\bm}$.
	
	By Proposition \ref{p:AltitudeIsPreserved}, $\alt(f(u_{m_i}))=\alt(u_{m_i})=\alt(g(u_{m_i}))$.
	Since cubes of positive dimensions in different components $\square^{n_j}$ of $\square^{{\vee}\bn}$ have different altitudes (see Example \ref{exa:Altitude}.(c)),
	both $f(u_{m_i})$ and $g(u_{m_i})$ lie in the same component $\square^{n_j}\subseteq \square^{{\vee}\bn}$.
	We have
	\[
		\bb|_{\square^{n_j}}(f(u_{m_i}))=
		\bb(f(u_{m_i}))=
		\ba(u_{m_i})=
		\bb(g(u_{m_i}))=
		\bb|_{\square^{n_j}}(g(u_{m_i}))
	\]
	But $\bb|_{\square^{n_j}}$ is injective, because $K$ is non-self-linked,
	which implies that $f(u_{m_i})=g(u_{m_i})$.
	Since every cube in $\square^{{\vee}\bm}$ is a face of some $u_{m_i}$, it follows that $f=g$.
	
	(b)
	The sequence of altitudes $\alt(c_1),\dotsc,\alt(c_l)$ has to be strictly increasing.

	(c)
	Let $\bb:\square^{{\vee}\bn}\to K$ be a cube chain
	and let $c\neq c'\in \square^{{\vee}\bn}[r]$.
	If $c$ and $c'$ belong to the same summand $\square^{n_i}$, then $\bb(c)\neq \bb(c')$ by injectivity of $\bb|_{\square^{n_i}}$.
	Otherwise, $\alt(c)\neq \alt(c')$, which also implies $\bb(c)\neq\bb(c')$.
	As a consequence, $\bb$ is an injective $\square$--map.
	 
	Now let $\ba:\square^{{\vee}\bm}\to K$ be a cube chain such that every $a_i$ is a face of $b_j$ for every $i\in\{1,\dotsc, l(\bm)\}$ and some $j=j(i)\in \{1,\dotsc,l(\bn)\}$.
	In particular, $\im(\ba)\subseteq\im(\bb)\subseteq K$.
	Since $\bb:\square^{{\vee}\bn} \to \im(\bb)$ is an isomorphism of $\square$--sets, the composition
	\[
		\square^{{\vee}\bm} \xrightarrow{\ba} \im(\bb) \xrightarrow{\bb^{-1}} \square^{{\vee}\bn}
	\]
	defines a morphism $\ba\to\bb$ in $\Ch(K)$.
\end{proof}

\subsection*{Length covering}
Let $K$ be an arbitrary bi-pointed $\square$--set.
Let $\preceq$ be the transitive and reflexive relation on the set of cubes of $K$ generated by
\[
	d^0_i(c)\preceq c\quad \text{and}\quad c\preceq d^1_i(c) 
\]
for all $n>0$, $c\in K[n]$ and  $i\in\{1,\dotsc,n\}$. We say that a cube $c\in K$ is \emph{accessible} if $\bO_K\preceq c$ and $c\preceq \bI_K$ for all $c\in K$. 
An elementary calculation shows that the set of accessible cubes forms a sub--$\square$--set $K_{acc}\subseteq K$,
which is non-empty if and only if $\bO_K\preceq \bI_K$.
If $K_{acc}=K$, then $K$ will be called \emph{accessible}.
Every accessible $\square$--set is connected and then admits at most one altitude function.
If an altitude function exists, then $c\preceq c'$ implies that $\alt(c)\leq \alt(c')$.

\begin{prp}\label{p:Acc}
	The inclusion $K_{acc}\subseteq K$ induces a homeomorphism $\vP(K_{acc})_\bO^\bI\cong \vP(K)_\bO^\bI$.
\end{prp}
\begin{proof}
	Every d-path admits a presentation (\ref{e:PathPresentation}) such that $\bO\preceq c_1\preceq \dotsm \preceq c_m\preceq \bI$.
\end{proof}

\emph{The length covering} of $K$ is a $\square$--set $\tilde{K}$ defined by
\begin{equation}
	\tilde{K}[n]=K[n]\times \Z,\qquad d^\varepsilon(c,h)=(d^\varepsilon_i(c), h+\varepsilon)
\end{equation}
for $c\in K[n]$, $\varepsilon\in\{0,1\}$, $h\in \Z$.
The projection $\tilde{K}\to K$ is a $\square$--map, and the projection $\tilde{K}\to\Z$ is an altitude function.
For $n\geq 0$, \emph{the length $n$ covering} of $K$ is the bi-pointed $\square$--set
\begin{equation}
	\tilde{K}_n=(\tilde{K},(\bO_K,0), (\bI_K,n))_{acc}.
\end{equation}

\begin{prp}\label{p:LenDecTilde}
	We have a sequence of homeomorphisms
	\[
		\vP(K)_{\bO_K}^{\bI_K}
		\cong
		\coprod_{n\geq 0} \vP(\tilde{K})_{(\bO_K,0)}^{(\bI_K,n)}
		\cong
		\coprod_{n\geq 0} \vP(\tilde{K_n})_{\bO_{\tilde{K}_n}}^{\bI_{\tilde{K}_n}}.
	\]
\end{prp}
\begin{proof}
	This follows from \cite[Proposition 5.3]{Raussen-Sim2} and Proposition \ref{p:Acc}.
\end{proof}

\begin{exa}
	For a finite strictly ordered set $A$ we have
	\[
		\tilde\square^A_n\cong
		\begin{cases}
			\square^A & \text{for $|A|=n$,}\\
			\emptyset & \text{otherwise.}
		\end{cases}
	\]
\end{exa}

\begin{exa}
	For cubes of $\tilde{Z}$ (see Example \ref{x:Z}) we will write $z^k_j$ for $(z^k,j)$. 
	For all $j,k$ we have
	\[
		z^{k}_j = d^0_1(z^{k+1}_j) \preceq z^{k+1}_j \preceq d^1_1(z^{k+1}_j) = z^k_{j+1}.
	\]
	If $0\leq j$ and $j+k\leq n$, then $z^k_j\in\tilde{Z}_n$, since 
	\[
		\bO_{\tilde{Z}_n}=z^0_0\preceq z^1_0\preceq z^0_1\preceq\dotsm \preceq z^0_j\preceq z^k_j\preceq z_{j+k}^0 \preceq z_{j+k}^1\preceq z_{j+k+1}^0 \preceq\dotsm\preceq z_n^0 =\bI_{\tilde{Z}_n}.
	\]
	Otherwise, $z^k_j\not\in \tilde{Z}_n$, since either $\alt(z^k_j)=j<0=\alt(\bO_{\tilde{Z}_n})$ or $\alt(d^1(z^k_j))=\alt(z_{j+k}^0)=j+k>n=\alt(\bI_{\tilde{Z}_n})$.
	Eventually, we have
	\begin{itemize}
	\item
		$\tilde{Z}_n[k]=\{z^k_0,\dotsc,z^k_{n-k}\}$,
	\item
		$d_i^\varepsilon(z^k_j)=z^{k-1}_{j+\varepsilon}$,
	\item
		$\bO_{\tilde{Z}_n}=z^0_0$, $\bI_{\tilde{Z}_n}=z^0_n$.
	\end{itemize}
\end{exa}

\obsol{
Let $\square\Set_*^*(n)\subseteq \square\Set_*^*$ denote the subcategory of accessible $\square$--sets with an altitude function such that $\alt(\bI)=n$.
\begin{prp}
	$\tilde{Z}_n$ is a final object in $\square\Set_*^*(n)$.
\end{prp}
\begin{proof}
	For $K\in\square\Set_*^*(n)$ there is a $\square$--map $f:K\to \tilde{Z}_n$, $f(c)=z^k_{\alt(c)}$,
	and this is the only attitude-preserving $\square$--map.
\end{proof}
}

\tikzset{zerocell/.style={circle, draw, minimum size=0.4cm, inner sep=0pt}}
\begin{figure}
	\centering
\begin{tikzpicture}[>=stealth',every node/.style={minimum size=0.4cm}]
		\node[zerocell] (000) at (0,0) {\scriptsize $0$};
		\node[zerocell] (001) at (3,0) {\scriptsize $1$};
		\node[zerocell] (010) at (0,3) {\scriptsize $1$};
		\node[zerocell, color=black!30!white] (100) at (2,1.4) {\scriptsize $1$};
		\node[zerocell] (110) at (2,4.4) {\scriptsize $2$};
		\node[zerocell] (101) at (5,1.4) {\scriptsize $2$};
		\node[zerocell] (011) at (3,3) {\scriptsize $2$};
		\node[zerocell] (111) at (5,4.4) {\scriptsize $3$};
		\path (000) edge (001) edge (010) edge[color=black!30!white] (100);
		\path (001) edge (011) edge (101);
		\path (010) edge (011) edge (110);
		\path (100) edge[color=black!30!white] (101) edge[color=black!30!white] (110);
		\path (111) edge (011) edge (101) edge (110);
		\node[below] at (1.5,0) {\scriptsize $0$};
		\node[left] at (0,1.5) {\scriptsize $0$};
		\node[above, color=black!30!white] at (1,0.7) {\scriptsize $0$};
		\node[below] at (4, 0.7)  {\scriptsize $1$};
		\node[left] at (3, 1.7)  {\scriptsize $1$};
		\node[above] at (1.5,3)  {\scriptsize $1$};
		\node[above] at (1,3.7)  {\scriptsize $1$};
		\node[left, color=black!30!white] at (2,2.7)  {\scriptsize $1$};
		\node[above, color=black!30!white] at (3.5,1.4)  {\scriptsize $1$};
		\node[above] at (3.5,4.4) {\scriptsize $2$};
		\node[right] at (5,2.9) {\scriptsize $2$};
		\node[above] at (4,3.7) {\scriptsize $2$};
	\end{tikzpicture}
	\caption{The geometric realization of $\tilde{Z}_3$. The vertices and the edges labeled with the same number are identified, as well as the three initial faces (containing the vertex 0) and the three final faces (containing the vertex 3).}
\end{figure}
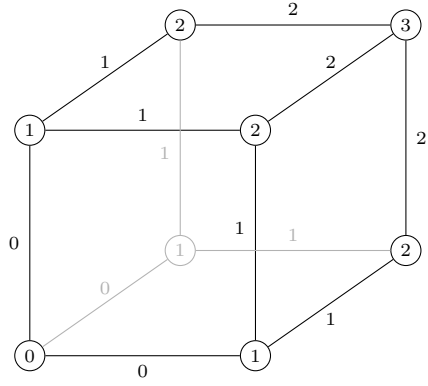

\subsection*{The ``covering" $\square$--set $Y^A$}

\newcommand\YCell[3]{(#1\;\parallel\; #2 \; \parallel\; #3)}

Let $A$ be a set having $n$ elements. Define a $\square$--set  $Y^A$ in the following way:
\begin{itemize}
\item
	$Y^A[k]$ is the set of pairs $(c,{<})$, where $c:A\to \{0,*,1\}$ is a function such that $|c^{-1}(*)|=k$ and ${<}$ is a total strict order on $c^{-1}(*)$.
\item
	If $(c,{<})\in Y^A[k]$ and $c^{-1}(*)=\{a_1<\dots< a_k\}$, then $d^\varepsilon_i(c,{<})=(c',{<}')$, where
\[
	c'(j)=
	\begin{cases}
		\varepsilon & \text{for $j=a_i$}\\
		c(j) & \text{otherwise,}
	\end{cases}
\]		
	and ${<}'$ is the restriction of ${<}$ to $(c')^{-1}(*)=c^{-1}(*)\setminus \{a_i\}$.
\item
	$\bO_{Y^A}=(\bO,\emptyset)$ and $\bI_{Y^A}=(\bI,\emptyset)$, where $\bO(a)=0$, $\bI(a)=1$ for all $a\in A$.
\end{itemize}
The  $\square$--set $Y^A$ is accessible and admits an altitude function $\alt(c,{<})=|c^{-1}(1)|$.

A $\square$--map $Y^A\to \tilde{Z}_n$ defined by
\begin{equation}\label{e:YZProj}
	p_A:Y^A[k]\ni (c,{<}) \mapsto z^k_{\alt(c,{<})}\in \tilde{Z}_n[k].
\end{equation}
is the only bi-pointed map from $Y^A$ to $\tilde{Z}_n$ (since the altitude of cubes must be preserved).

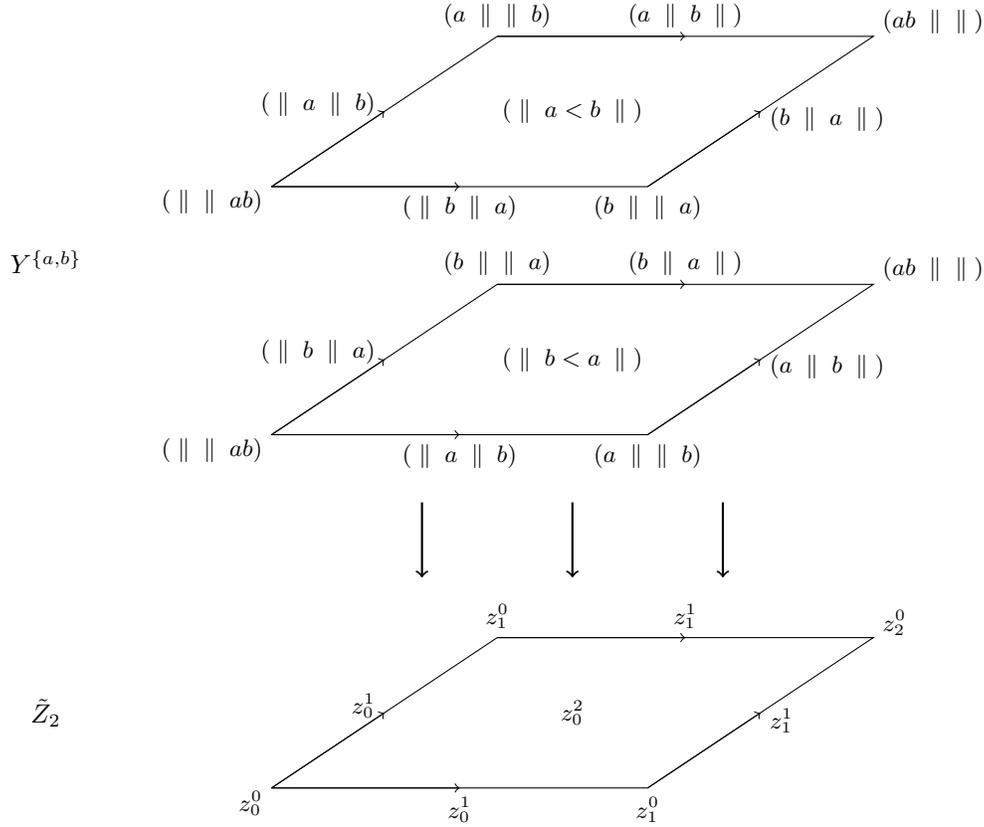
\begin{figure}
\begin{tikzpicture}
	\def\y{6};
	\draw (0,\y+0) -- (5,\y+0) -- (8,\y+2) -- (3,\y+2) -- (0,\y+0);
	\draw[->] (0,\y+0)--(2.5,\y+0);
	\draw[->] (0,\y+0)--(1.5,\y+1);
	\draw[->] (3,\y+2)--(5.5,\y+2);
	\draw[->] (5,\y+0)--(6.5,\y+1);
	\node[below] at(2.5,\y+0) {\small $\YCell{}{b}{a}$};
	\node[above] at(5.5,\y+2) {\small $\YCell{a}{b}{}$};
	\node[left] at(1.5,\y+1.1) {\small $\YCell{}{a}{b}$};
	\node[right] at(6.5,\y+0.9) {\small $\YCell{b}{a}{}$};
	\node at(4,\y+1) {\small $\YCell{}{a<b}{}$};	
	\node[left] at(0,\y-0.2) {\small $\YCell{}{}{ab}$};
	\node[right] at(8,\y+2.2) {\small $\YCell{ab}{}{}$};
	\node[above] at(3,\y+2) {\small $\YCell{a}{}{b}$};
	\node[below] at(5,\y+0) {\small $\YCell{b}{}{a}$};
	\def\y{2.7}
	\draw (0,\y+0) -- (5,\y+0) -- (8,\y+2) -- (3,\y+2) -- (0,\y+0);
	\draw[->] (0,\y+0)--(2.5,\y+0);
	\draw[->] (0,\y+0)--(1.5,\y+1);
	\draw[->] (3,\y+2)--(5.5,\y+2);
	\draw[->] (5,\y+0)--(6.5,\y+1);
	\node[below] at(2.5,\y+0) {\small $\YCell{}{a}{b}$};
	\node[above] at(5.5,\y+2) {\small $\YCell{b}{a}{}$};
	\node[left] at(1.5,\y+1.1) {\small $\YCell{}{b}{a}$};
	\node[right] at(6.5,\y+0.9) {\small $\YCell{a}{b}{}$};
	\node at(4,\y+1) {\small $\YCell{}{b<a}{}$};	
	\node[left] at(0,\y-0.2) {\small $\YCell{}{}{ab}$};
	\node[right] at(8,\y+2.2) {\small $\YCell{ab}{}{}$};
	\node[above] at(3,\y+2) {\small $\YCell{b}{}{a}$};
	\node[below] at(5,\y+0) {\small $\YCell{a}{}{b}$};
	\def\y{-2}
	\draw (0,\y+0) -- (5,\y+0) -- (8,\y+2) -- (3,\y+2) -- (0,\y+0);
	\draw[->] (0,\y+0)--(2.5,\y+0);
	\draw[->] (0,\y+0)--(1.5,\y+1);
	\draw[->] (3,\y+2)--(5.5,\y+2);
	\draw[->] (5,\y+0)--(6.5,\y+1);
	\node[below] at(2.5,\y+0) {\small $z^1_0$};
	\node[above] at(5.5,\y+2) {\small $z^1_1$};
	\node[left] at(1.5,\y+1.1) {\small $z^1_0$};
	\node[right] at(6.5,\y+0.9) {\small $z^1_1$};
	\node at(4,\y+1) {\small $z^2_0$};	
	\node[left] at(0,\y-0.2) {\small $z^0_0$};
	\node[right] at(8,\y+2.2) {\small $z^0_2$};
	\node[above] at(3,\y+2) {\small $z^0_1$};
	\node[below] at(5,\y+0) {\small $z^0_1$};
	\node at(-3,5) {$Y^{\{a,b\}}$};
	\node at(-3,-1) {$\tilde{Z}_2$};
	\draw[->,thick] (4,1.8)--(4,0.8);
	\draw[->,thick] (2,1.8)--(2,0.8);
	\draw[->,thick] (6,1.8)--(6,0.8);
\end{tikzpicture}
\caption{The map $|Y^{\{a,b\}}|\to|\tilde{Z}_2|$. 
The cells having the same labels are identified.}
\end{figure}

The symmetric group $\Sigma_A$ acts on $Y^A$ on the right by $(c,<)\sigma=(c\circ \sigma, {<}\sigma)$, where $a({<}\sigma)b$ iff $\sigma(a)<\sigma(b)$. 
Since both the initial $\bO_{Y^A}$ and the final vertex $\bI_{Y^A}$ are fixed points,
then $\Sigma_A$ acts on $Y^A$ regarded as a bi-pointed $\square$--set.
This action induces a $\Sigma_A$--action on $|Y^A|$ and then, on $\vP(Y^A)_\bO^\bI$.

\begin{prp}\label{p:OrbitMapIsIso}
	The $\square$--map $Y^A\to \tilde{Z}_n$ is $\Sigma_A$--invariant. 
	The induced orbit $\square$--map $Y^A/\Sigma_A \to \tilde{Z}_n$ is an isomorphism.
\end{prp}
\begin{proof}
	The first statement is obvious.
	If $p_A(c,{<})=z^k_j=p_A(c',{<}')$, then $c^{-1}(\varepsilon)$ and $(c')^{-1}(\varepsilon)$ are equipotent for $\varepsilon\in\{0,*,1\}$.
	Thus, there exists a permutation $\sigma\in\Sigma_A$ such that $c\circ \sigma=c'$ and $\sigma$ agrees with the orders on $c^{-1}(*)$ and $(c')^{-1}(*)$.
\end{proof}

The $n$--cubes of $Y^A$ are in 1--1 correspondence with total strict orders $<$ on $A$.
For every such order $<$ there is a $\square$--map
\begin{equation}\label{e:OrderIncl}
	\square^{(A,{<})}\ni c \mapsto (c, {<}|_{c^{-1}(*)})\in Y^A,
\end{equation}
which is injective.
The images of these maps, for all possible choices of strict orders on $A$, cover $Y^A$.
Applying Proposition \ref{p:NSLCov} and Lemma \ref{l:ChNSL} we obtain:

\begin{prp}\label{p:YANonSelfLinked}
	$Y^A$ is non-self-linked, and its cube chain category $\Ch(Y^A)$ is a poset.\qed
\end{prp}

The following statement is elementary and will be used later on.
\begin{lem}\label{l:YAFace}
	Let $(c_1,{<_1})\in Y^A[k_1]$, $(c_2,{<_2})\in Y^A[k_2]$. The following conditions are equivalent.
	\begin{itemize}
	\item
			$(c_1,{<_1})$ is a face of $(c_2,{<_2})$,
		\item
			$c_1^{-1}(*)\subseteq c_2^{-1}(*)$, $c_2^{-1}(0)\subseteq c_1^{-1}(0)$, $c_2^{-1}(1)\subseteq c_1^{-1}(1)$, and ${<_1}={<_2}|_{c_1^{-1}(*)}$.\qed		
	\end{itemize}
\end{lem}

For a cell $(c,{<})\in Y^A[k]$, $c^{-1}(*)=\{a_1<\dotsm< a_k\}$, we will use the notation 
\begin{equation}\label{e:YCellNotation}
	\YCell{c^{-1}(1)}{a_1<\dotsm<a_k}{c^{-1}(0)}.
\end{equation}
Note that $\bO_{Y^A}=\YCell{\emptyset}{-}{A}$, $\bI_{Y^A}=\YCell{A}{-}{\emptyset}$,
\begin{align}\label{e:dYAO}
	d^0_V(\YCell{A_1}{a_1<\dotsm<a_k}{A_0})
	&=
	\YCell{A_1}{a_{i(1)}<\dotsm<a_{i(m)}}{A_0\cup \{a_j\}_{j\in V}}\\
	d^1_V(\YCell{A_1}{a_1<\dotsm<a_k}{A_0}) \label{e:dYAI}
	&=
	\YCell{A_1\cup \{a_j\}_{j\in V}}{a_{i(1)}<\dotsm<a_{i(m)}}{A_0}
\end{align}
for $\YCell{A_1}{a_1<\dotsm<a_k}{A_0}\in Y^A[k]$, $V\subseteq \{1,\dotsc,k\}$, and
\[
	\{1,\dotsc,k\}\setminus V=\{i(1)<\dotsm<i(m)\}.
\]
For $\sigma\in \Sigma_A$ we have
\begin{equation}
	\YCell{A_1}{a_1<\dotsm<a_k}{A_0}\sigma
	=
	\YCell{\sigma^{-1}(A_1)}{\sigma^{-1}(a_1)<\dotsm<\sigma^{-1}(a_k)}{\sigma^{-1}(A_0)}.
\end{equation}

For every $a\in A$ there is a ``valuation" d-map $v_a:|Y^A|\to \vI$ given by
\begin{equation}\label{e:VMap}
	v_a([\YCell{A_1}{a_1<\dotsm<a_k}{A_0}; x_1,\dotsc,x_k])=
	\begin{cases}
		0 & \text{for $a\in A_0$,}\\
		x_j & \text{for $a=a_j$,}\\
		1 & \text{for $a\in A_1$.}
	\end{cases}
\end{equation}
The product $v:|Y^A|\to \vI^A$ of these maps is a bi-pointed $\Sigma_A$--equivariant d-map,
which is not the geometric realization of a $\square$--map.

\obsol{
}

\section{On covering $\vP(Y^A)\to\vP(Z)$}

Fix a set $A$ having $n$ elements. 
The main goal of this section is the following:

\begin{prp}\label{p:MainCov}
 The $\square$--map $p_A:Y^A\to \tilde{Z}_n$ induces a $\Sigma_A$--principal bundle
\[
	\pi:\vP(Y^A)_\bO^\bI\to  \vP(\tilde{Z}_n)_\bO^\bI.
\] 
\end{prp}

\begin{prp}\label{p:YActionFree}
	The action of $\Sigma_A$ on $\vP(Y^A)_\bO^\bI$ is free.
\end{prp}
\begin{proof}
	Let $\alpha\in\vP(Y^A)_\bO^\bI$,
	$\sigma\in \Sigma_A$
	and assume that $\alpha=\alpha\sigma$.
	For $a\in A$, the composition $v_a\circ \alpha$ is a d-path from $0$ to $1$ in $\vI$,
	and there exists $t\in (0,1)$ such that $0<v_a(\alpha(t))<1$.
	Hence there is $j\in\{1\dotsc,k\}$ such that $a=a_j$,
	where 
	$
		\carr(\alpha(t))=\YCell{A_1}{a_1<\dotsm<a_k}{A_0}.
	$
	But 	
	\[
		\carr(\alpha(t))=\carr((\alpha\sigma)(t))=\carr(\alpha(t))\sigma=\YCell{\sigma^{-1}(A_1)}{\sigma^{-1}(a_1)<\dotsm<\sigma^{-1}(a_k)}{\sigma^{-1}(A_0)}.
	\]
	Eventually, $a=a_j=\sigma^{-1}(a_j)=\sigma^{-1}(a)$ for all $a\in A$.
\end{proof}

It remains to show that $\pi$ is a quotient map and that the preimages of points of $\vP(\tilde{Z}_n)$ are single orbits.
We will study the commutative diagram
\begin{equation}\label{e:CovDiag}
	\begin{diagram}
		\node{Y^A[n]\times \vP(\vI^n)_\bO^\bI}
			\arrow[3]{e,t}{\mu:(c,\beta)\mapsto [c;\beta]}
			\arrow{s,l}{\varrho:(c,\beta)\mapsto \beta}
		\node{}
		\node{}
		\node{\vP(Y^A)_\bO^\bI}
			\arrow{s,r}{\pi}
	\\
		\node{\vP(\vI^n)_\bO^\bI}
			\arrow[3]{e,t}{\nu:\beta\mapsto [z^n_0;\beta]}
		\node{}
		\node{}
		\node{\vP(\tilde{Z}_n)_\bO^\bI}
	\end{diagram}
\end{equation}

The left-hand spaces can be regarded as spaces of presentations (\ref{e:PathPresentation}) of d-paths belonging to the corresponding right-hand spaces.
This is a diagram of $\Sigma_A$--spaces; the action is free on the top spaces and trivial on the bottom spaces.

\begin{prp}\label{p:YPres}
	The map $\mu$ is surjective, ie, 
	every d-path $\alpha\in \vP(Y^A)_\bO^\bI$ has a presentation having the form $[c;\beta]$, where $c\in Y^A[n]$, $\beta\in\vP(\vI^n)_\bO^\bI$.
\end{prp}
\begin{proof}
Let 
\begin{equation}
		\alpha = [c_1;\beta_1]\overset{t_1}{*}[c_2;\beta_2]\overset{t_2}{*}\dotsm\overset{t_{m-1}}{*}[c_m;\beta_m]
\end{equation}
be a presentation from Proposition \ref{p:DPathPres} with a minimal possible $m$.
We will assume that $m>1$ and obtain a contradiction.

Denote $p=\dim(c_1)$, $q=\dim(c_2)$.
Conditions (c) and (d) in Proposition \ref{p:DPathPres} imply that
there exist $V\subseteq \{1,\dotsc,p\}$, $W\subseteq \{1,\dotsc,q\}$ and $\bx\in\vI^{p-|V|}=\vI^{q-|W|}$
such that
\[
d^1_{V}(c_1)=d^0_{W}(c_2),\qquad
\beta_1(t_1)=\delta^1_{V}(\bx), \qquad
\beta_2(t_1)=\delta^0_{W}(\bx).
\]
Moreover, $d^0(c_1)=\bO$.
Denote $r=p-|V|=q-|W|$, $s=p+|W|=q+|V|$, $c''=d^1_{V}(c_1)=d^0_{W}(c_2)$.

Using equations (\ref{e:dYAO}) and (\ref{e:dYAI}) we obtain that for some subsets $B_1,B_2\subseteq A$ we have
\begin{align*}
	c_1&=(\emptyset\;\parallel\; b^1_1<\dotsm<b^1_p\;\parallel\; A\setminus B_1)\in Y^A[p]\\
	c_2&=( B_1\setminus B_2 \;\parallel\; b^2_1<\dotsm<b^2_q \;\parallel\; A\setminus (B_1\cup B_2))\in Y^A[q]\\
	c''&=(B_1\setminus B_2\;\parallel\; e_1<\dotsm<e_r\;\parallel\;A\setminus B_1)\in Y^A[r],
\end{align*}
where
\[
	B_1=\{b^1_1,\dotsc,b^1_p\},\;B_2=\{b^2_1,\dotsc,b^2_q\},\; B_1\cap B_2=\{e_1,\dotsc,e_r\},
\]
\[
	V=\{i\;|\; b^1_i\not\in B_2\}, \quad W=\{i\;|\; b^2_i\not\in B_1\},
\]
and the order on $B_1\cap B_2$ is the common restriction of the orders on $B_1$ and $B_2$.
Let
\[
	c'=(\emptyset\;\parallel\; a_1<\dotsm<a_s\;\parallel\;A\setminus(B_1\cup B_2))\in Y^A[s]
\]
where $B_1\cup B_2=\{a_1,\dotsc,a_s\}$ and the order above is a common extension of the orders on $B_1$ and $B_2$;
this extension is not necessarily unique.
Let 
\[
	V'=\{i\;|\; a_i\not\in B_2\},\quad W'=\{i\;|\;a_i\not\in B_1\}.
\]
Obviously $c_1=d^0_{W'}(c')$ and $c_2=d^1_{V'}(c')$; thus
\begin{align*}
	[c_1;\beta_1(t)]&=[d^0_{W'}(c');\beta_1(t)]=[c';\delta^0_{W'}(\beta_1(t))]\quad  \text{for $t\in[0,t_1]$},\\
	[c_2;\beta_2(t)]&=[d^1_{V'}(c');\beta_2(t)]=[c';\delta^1_{V'}(\beta_2(t))]\quad  \text{for $t\in[t_1,t_2]$}.
\end{align*}
We have 
\[
	d^1_V d^0_{W'}(c')=d^1_V(c_1)=c''=d^0_W(c_2)= d^0_W d^1_{V'}(c').
\]
Since $Y^A$ is non-self-linked, $d^1_V d^0_{W'}=d^0_W d^1_{V'}$ is a precubical identity and thus, $\delta^0_{W'} \delta^1_V =\delta^1_{V'} \delta^0_W$.
As a consequence,
\[
	\delta^0_{W'}(\beta_1(t_1))=
	\delta^0_{W'}(\delta^1_V(\bx))=	
	\delta^1_{V'}(\delta^0_W(\bx))=
	\delta^1_{V'}(\beta_2(t_1)).
\]
Thus, $\alpha$ admits a shorter presentation
\[
	\alpha=[c'; (\delta^0_{W'}\beta_1)\overset{t_1}{*} (\delta^1_{V'}\beta_2)]\overset{t_2}{*}\dotsm\overset{t_{m-1}}{*}[c_m;\beta_m].\qedhere
\]
\end{proof}

\begin{rem}
	There is an alternative way to prove Proposition \ref{p:YPres}.
	One can show that for every $\alpha\in\vP(Y^A)_\bO^\bI$
	there exists a total order $<$ on $A$ that extends the order on $\carr(\alpha(t))^{-1}(*)$ for every $t\in[0,1]$.
	For such an order we have $\alpha=[(\const_*,<);v\circ \alpha]$.
	This argument seems more intuitive but we are not able to find a similar proof for $\tilde{Z}^n$.
\end{rem}

To prove a similar statement for the bottom map $\nu$ we need the following.
\begin{lem}\label{l:FaceSwap}
	Assume that $p,q\geq 0$, $V\subseteq\{1,\dotsc,p\}$, $W\subseteq\{1,\dotsc,q\}$ and $p-\left|V\right|=q-|W|$.
	Then there exist $V',W'\subseteq \{1,\dots,s\}$, $s=p+|W|=q+|V|$, $|V'|=|V|$, $|W|=|W'|$,
	for which $d^1_V d^0_{W'}=d^0_W d^1_{V'}$ is a precubical identity.
\end{lem}
\begin{proof}
	Induction with respect to $s$.
	It is obvious if either $p=0$ or $q=0$ so we assume otherwise.
	Consider the following cases.
	\begin{enumerate}
	\item
		$q\in W$.
		There exist $V'',W''\subseteq \{1,\dotsc,s-1\}$ such that $d^1_V d^0_{W''}=d^0_{W\setminus\{q\}} d^1_{V''}$. Let $V'=V''$, $W'=W''\cup\{s\}$. Then
		\[
			d^1_V d^0_{W'}
			=
			d^1_V d^0_{W''}d^0_s
			=
			d^0_{W\setminus\{q\}} d^1_{V''}d^0_s
			=
			d^0_{W\setminus\{q\}} d^0_{s-|V''|} d^1_{V''}
			=
			d^0_{W\setminus\{q\}} d^0_{q} d^1_{V'}
			=
			d^0_{W} d^1_{V'}.			
		\]
	\item
		$q\not\in W$ and $p\in V$.
		The argument is similar.
	\item
		$s>0$, $q\not\in W$, $p\not\in V$.
		Then $V\subseteq \{1,\dotsc,p-1\}$, $W\subseteq \{1,\dotsc,q-1\}$, so there exist $V',W'\subseteq \{1,\dots,s-1\}$ satisfying the required property.\qedhere
	\end{enumerate}
\end{proof}

\begin{prp}\label{p:ZPres}
	The map $\nu$ {\normalfont (\ref{e:CovDiag})} is surjective, ie, 
	every d-path $\alpha\in \vP(\tilde{Z}_n)_\bO^\bI$ has a presentation having the form $[z^n_0;\beta]$ where $\beta\in\vP(\vI^n)_\bO^\bI$.
\end{prp}
\begin{proof}
	We slightly modify the proof of  \ref{p:YPres}. 
	For a minimal presentation $\alpha=[c_1;\beta_1]\overset{t_1}*\dotsm\overset{t_{m-1}}*[c_m;\beta_m]$ such that $m>1$
	 we have $c_1=z^p_0$, $c_2=z^q_{p-r}$
	  and there exist $\bx\in\vI^{r}$, $V\subseteq \{1,\dotsc,p\}$, $W\subseteq \{1,\dotsc,q\}$
	  such that $\delta^1_V(\bx)=\beta_1(t_1)$ and $\delta^0_W(\bx)=\beta_2(t_1)$.
	By Lemma \ref{l:FaceSwap}, we can find subsets $V',W'\subseteq \{1,\dotsc,s\}$, $s=p+q-r$, that satisfy the precubical identity $d^1_V d^0_{W'}=d^0_W d^1_{V'}$.
	As a consequence, the first two terms of the presentation can be merged into $[z^s_0; (\delta^0_{W'}\beta_1)*(\delta^1_{V'}\beta_2)]$: a contradiction.
	Thus, $m=1$ and $c_1=z^n_0$.
\end{proof}

\begin{prp}\label{p:PiIsQuotient}
	The map $\pi:\vP(Y^A)_\bO^\bI\to \vP(\tilde{Z}_n)_\bO^\bI$ is a quotient map.
\end{prp}
\begin{proof}
	The map $\pi$ is surjective since $\nu$ and $\varrho$ are surjective.

	By Proposition \ref{p:NatProd}, $\pi$ is a product of the map $\pi_N:\vN_{[0,n]}(Y^A)\to \vN_{[0,n]}(\tilde{Z}_n)$ between the spaces of natural paths and the identity on $\vP([0,n])_0^n$. 

	The space $\vN_{[0,n]}(\vI^n)_\bO^\bI$ is compact \cite[Proposition 9.4]{Z-Cub2}, and then  $\vN_{0,n}(Y^A)_\bO^\bI$ is also compact as an image of $Y^A\times \vN_{[0,n]}(\vI^n)_\bO^\bI$ (Proposition \ref{p:YPres}).
	Since $\pi_N$ is a surjective map from a compact space, it is a quotient map and so is $\pi$.
\end{proof}

In the remaining part of this section we show that $\pi^{-1}(\alpha)$ is a single orbit for every $\alpha\in\vP(\tilde{Z}_n)_\bO^\bI$.

\emph{The support} of a d-path $\beta\in \vP(\vI)_0^1$ is the interval
\begin{equation}\label{e:Support}
	\supp(\beta) = \{t\in [0,1]\;|\; 0<\beta(t)<1\}.
\end{equation}
Let $\sim$ be the equivalence relation on $\vP(\vI^n)_\bO^\bI$ spanned by
\begin{equation}\label{e:EqZ}
		(\beta_1,\dotsc,\beta_i,\beta_{i+1},\dotsc,\beta_n)
		\sim 
		(\beta_1,\dotsc,\beta_{i+1},\beta_{i},\dotsc,\beta_n)
\end{equation}
for $\supp(\beta_i)\cap \supp(\beta_{i+1})=\emptyset$.
We say that $\beta=(\beta_1,\dotsc,\beta_n)\in \vP(\vI^n)_\bO^\bI$ is \emph{sorted}
if for every $i$ such that $\supp(\beta_{i})\cap \supp(\beta_{i+1})=\emptyset$, 
$\supp(\beta_i)$ precedes $\supp(\beta_{i+1})$,
ie, the exists $t\in(0,1)$ such that $\beta_i(t)=1$ and $\beta_{i+1}(t)=0$.

\begin{prp}\label{p:ZPresEq}
	Let $\beta,\beta'\in\vP(\vI^n)$.
	\begin{enumerate}[\normalfont (a)]
	\item
		If $\beta\sim \beta'$, then $\nu(\beta)=\nu(\beta')$.
	\item
		For every $\beta\in\vI^n$ there exists $\beta'\in\vP(\vI^n)_0^1$ such that $\beta\sim\beta'$ 
		and $\beta'$ is sorted.
	\item
		If $\nu(\beta)=\nu(\beta')$, then $\beta\sim \beta'$.
	\end{enumerate}
\end{prp}

We need the following lemma:
\begin{lem}\label{l:DimRed}
	Let $y,x_1,\dotsc,x_{n-1},x'_1,\dotsc,x'_{n-1}\in [0,1]$.
	If
	\[
		[z^n_0;y,x_1,\dotsc,x_{n-1}]=[z^n_0;y,x'_1,\dotsc,x'_{n-1}]\in \tilde{Z}_n,
		\]
		then $[z^{n-1}_0;x_1,\dotsc,x_{n-1}]=[z^{n-1}_0;x'_1,\dotsc,x'_{n-1}]\in\tilde{Z}_{n-1}$.
\end{lem}
\begin{proof}
	Elementary calculation.
\end{proof}

\begin{proof}[Proof of \ref{p:ZPresEq}]
	To prove (a) it is enough to consider the case where $\beta$ and $\beta'$ are as in (\ref{e:EqZ}).
	If $\supp(\beta_i)\cap \supp(\beta_{i+1})=\emptyset$,
	then for every $t\in[0,1]$ at least one of numbers $\beta_i(t)$, $\beta_{i+1}(t)$ is equal to either $0$ or $1$.
	If $\beta_i(t)=\varepsilon\in\{0,1\}$, then
\begin{multline*}
	\nu(\beta)(t)
	=
	[z^n_0;\beta_1(t),\dotsc,\beta_{i-1}(t),\beta_i(t),\beta_{i+1}(t),\dotsc,\beta_n(t))]\\
	=
	[z^{n-1}_{\varepsilon}; (\beta_1(t),\dotsc,\beta_{i-1}(t),\beta_{i+1}(t),\dotsc,\beta_n(t))]\\
	=
	[z^n_0; (\beta_1(t),\dotsc, \beta_{i-1}(t),\beta_{i+1}(t),\beta_{i}(t),\dotsc,\beta_n(t))]
	=\nu(\beta')(t),
\end{multline*}
since $d^\varepsilon_i(z^n_0)=z^{n-1}_\varepsilon=d^\varepsilon_{i+1}(z^n_0)$.

Point (b) follows immediately from the definition.

To prove (c) assume that $\alpha=\nu(\beta)=\nu(\beta')\in \vP(\tilde{Z}_n)_\bO^\bI$.
By (a) and (b) we can assume that both $\beta$ and $\beta'$ are sorted.
We will prove that $\beta=\beta'$ by induction with respect to $n$.
For $t\in [0,1]$ let 
\[
	\alpha(t)=[z^{k(t)}_{l(t)}; \gamma_{1}(t),\dotsc,\gamma_{k(t)}(t)]
\]
be the canonical presentation.
$\gamma_1$ is a function with domain $\bigcup_{i} \supp(\beta_i)=\bigcup_{i} \supp(\beta'_i)$;
for $t\in\dom(\gamma_1)$ we have
\[
	\gamma_1(t)
	= \beta_{\min\{i\st 0<\beta_i(t)<1\}}(t)
	= \beta'_{\min\{i\st 0<\beta'_i(t)<1\}}(t),
\]
and $\alpha(t)$ is a vertex  for $t\not\in \dom(\gamma_1)$.

Note that for any pair of numbers $0\leq s_0<s_1\leq 1$ the following conditions are equivalent:
\begin{enumerate}
\item
	$\gamma_1$ is defined and increasing on $(s_0,s_1)$,
	$\lim_{t\to s_0^+} \gamma_1(t)=0$ and $\lim_{t\to s_1^-}\gamma_1(t)=1$.
\item
	There exists $k\in\{1,\dots,n\}$ such that $\supp(\beta_k)=(s_0,s_1)$ and $\supp(\beta_k)\cap\supp(\beta_l)=\emptyset$ for every $l<k$.
\item
	There exists $k'\in\{1,\dots,n\}$ such that $\supp(\beta'_{k'})=(s_0,s_1)$ and $\supp(\beta'_{k'})\cap\supp(\beta'_{l'})=\emptyset$ for every $l'<k'$.
\end{enumerate}
Obviously such $k$ and $k'$ are unique for a given $(s_0,s_1)$.
Intervals $(s_0,s_1)$ that satisfy (1)--(3) will be called distinguished intervals.
There exists at least one distinguished interval: $\supp(\beta_1)$ is an example.

Let $(s_0,s_1)$ be a distinguished interval for which $k>1$.
Then $\supp(\beta_{k-1})<\supp(\beta_k)$ (since $\beta$ is sorted and (2) is assumed).
Also, $\supp(\beta_{k-2})$ is disjoint with $\supp(\beta_k)$(by (2))
and $\supp(\beta_{k-2})<\supp(\beta_k)$:
in the opposite case we have $\supp(\beta_{k-1})<\supp(\beta_k)<\supp(\beta_{k-2})$,
which implies that $\beta$ is not sorted.
By repeating this argument we show that $\supp(\beta_1)<\supp(\beta_k)$.
As a consequence, the left-most distinguished interval is $\supp(\beta_1)$.
The same argument applies for $\beta'$, so that $\supp(\beta_1')=\supp(\beta_1)$.
As a consequence,
	\[
		\beta_1(t)=\beta'_1(t)=
		\begin{cases}
			0 & \text{for $t<\supp(\beta_1)=\supp(\beta'_1)$,}\\
			\gamma_1(t) & \text{for $t\in \supp(\beta_1)$,}\\
			1 & \text{for $t>\supp(\beta_1)$.}
		\end{cases}
	\]
	Since $[z^n_0;\beta]=\nu(\beta)=\nu(\beta')=[z^n_0;\beta']$, from Lemma \ref{l:DimRed} we obtain that 
	\[
		[z^{n-1}_0;\beta_2,\dots,\beta_n]=[z^{n-1}_0;\beta'_2,\dots,\beta'_n].
	\]
The inductive hypothesis implies that $(\beta_2,\dotsc,\beta_n)=(\beta'_2,\dotsc,\beta'_n)$, and then, $\beta=\beta'$.
\end{proof}

Let $\sim$ be the equivalence relation on $Y^A\times \vP(\vI^n)_\bO^\bI$  spanned by 
\begin{multline}\label{e:PresSwap}
	((\emptyset\;\parallel\; a_1<\dotsm<a_i<a_{i+1}<\dotsc <a_n \;\parallel\;\emptyset), (\beta_1,\dotsc,\beta_i,\beta_{i+1},\dotsc,\beta_n))\\
	\sim
	((\emptyset\;\parallel\; a_1<\dotsm<a_{i+1}<a_{i}<\dotsc <a_n \;\parallel\;\emptyset), (\beta_1,\dotsc,\beta_{i+1},\beta_{i},\dotsc,\beta_n))
\end{multline}
for all $i$, $a_i$ and $\beta_i$ such that the supports of $\alpha_i$ and $\alpha_{i+1}$ are disjoint.

\begin{prp}\label{p:YPresEq}
	Let $c,c'\in Y^A[n]$, $\beta,\beta'\in \vP(\vI^n)_\bO^\bI$.
	Assume that $(c;\beta)\sim(c';\beta')$. Then
	\begin{enumerate}[\normalfont (a)]
	\item
		If $(c;\beta)\sim(c';\beta')$, then $\mu(c,\beta)=\mu(c',\beta')$.
	\item
		If $(c;\beta)\sim(c';\beta')$, then $\beta\sim \beta'$.
	\item
		If $\beta\sim\beta'$, then there exists $c''\in Y^A[n]$ such that $(c,\beta)\sim (c'',\beta')$.
	\end{enumerate}
\end{prp}
\begin{proof}
	Proof of (a) is similar to Proof of \ref{p:ZPresEq}.(a). Points (b) and (c) follow immediately from the definitions.
\end{proof}

\begin{prp}\label{p:PiInj}
	Let $\alpha,\alpha'\in \vP(Y^A)_\bO^\bI$. If $\pi(\alpha)=\pi(\alpha')$, then there exists $\sigma\in\Sigma_A$ such that  $\alpha=\alpha'\sigma$.
\end{prp}
\begin{proof}
	Choose $c,c'\in Y^A[n]$, $\beta,\beta'\in\vP(\vI^n)_\bO^\bI$ such that $\alpha=\mu(c,\beta)$, $\alpha'=\mu(c',\beta')$ (Prop. \ref{p:PiIsQuotient}).
	From commutativity of (\ref{e:CovDiag}), $\pi(\alpha)=\nu(\beta)=\nu(\beta')=\pi(\alpha')$;
	therefore $\beta\sim \beta'$ (Prop. \ref{p:ZPresEq}.(c)).
	There exists $c''\in Y^A[n]$ such that $(c,\beta)\sim(c'',\beta')$ (Prop. \ref{p:YPresEq}.(c))
	and then $\mu(c,\beta)=\mu(c'',\beta')$ (Prop. \ref{p:YPresEq}.(a)).
	The action of $\Sigma_A$ on $Y^A[n]$ is transitive, so there exists $\sigma\in \Sigma_A$ such that $c''= c'\sigma$.
	Eventually,
	\[
		\alpha=\mu(c,\beta)=\mu(c'',\beta')=\mu(c'\sigma,\beta')=\mu(c',\beta')\sigma=\alpha'\sigma.\qedhere
	\]
\end{proof}

\begin{proof}[Proof of \ref{p:MainCov}]
	Consider the commutative diagram
	\[
		\begin{diagram}
			\node{\vP(Y^A)_\bO^\bI}
				\arrow{s,l}{p}
				\arrow{se,t}{\pi}
		\\
			\node{\vP(Y^A)_\bO^\bI/\Sigma_A}
				\arrow{e,t}{i}
			\node{\vP(Y^A/\Sigma_A)_\bO^\bI\rlap{$\;\cong\vP(\tilde{Z}_n)_\bO^\bI$}}
		\end{diagram}
	\]
	The map $p$ is a $\Sigma_A$--covering (Proposition \ref{p:YActionFree}).
	The map $i$ is injective (Proposition \ref{p:PiInj}) and surjective (since $\pi$ is surjective, Proposition \ref{p:PiIsQuotient}).
	But $\pi$ is a quotient map, so then $i$ is a homeomorphism.
\end{proof}

We conclude this section with two applications of Proposition \ref{p:MainCov}.

\begin{prp}\label{p:ZHtpCW}
	$\vP(Y^A)_\bO^\bI$ has the $\Sigma_A$--homotopy type of a (free) $\Sigma_A$--CW-complex, and
	$\vP(\tilde{Z}_n)_\bO^\bI$ has the homotopy type of a CW-complex.
\end{prp}
\begin{proof}
	Raussen \cite{Raussen-Trace} shows that for a non-self-linked $\square$--set $K$, the space $\vP(K)_\bO^\bI$ is equilocally convex.
	His construction is invariant with respect to automorphisms of $K$.
	As a consequence, $\vP(Y^A)_\bO^\bI$ is $\Sigma_A$--equilocally convex \cite[Definition 4.2]{Waner-Milnor}.
	By the equivariant version of Milnor's Theorem \cite[Theorem 4.9]{Waner-Milnor},
	$\vP(Y^A)_\bO^\bI$ has the homotopy type of a CW-complex.
	Now Proposition \ref{p:MainCov} implies that $\vP(\tilde{Z}_n)_\bO^\bI$ 
	has the homotopy type of a CW-complex.
\end{proof}

\begin{prp}\label{p:ChPYEqHE}
	The spaces $\vP(Y^A)_\bO^\bI$ and $|\Ch(Y^A)|$ are $\Sigma_A$--homotopy equivalent.
\end{prp}
\begin{proof}
	By \cite[Corollary 6.6]{Z-Cub2}, the maps
	\[
		\vN^t_{[0,n]}(K)_\bO^\bI \xrightarrow{\subseteq} \vP_{[0,n]}(K;n)_\bO^\bI \xrightarrow{\simeq} \vP(K;n)_\bO^\bI
	\]
	are homotopy equivalences with homotopy inverses that are functorial the respect to $K$.
	Thus, $\vN^t_{[0,n]}(Y^A)_\bO^\bI$ and $\vP(Y^A)_\bO^\bI$
	are $\Sigma_A$--homotopy equivalent
	and they have the homotopy type of a $\Sigma_A$--CW-complex (Proposition \ref{p:ZHtpCW}).
	Here $\vN^t_{[0,n]}(K)_\bO^\bI$ denotes the space of natural tame d-paths
	(see \cite[2.9]{Z-Cub2}).
	(Actually, all d-paths on $Y^A$ are tame, so that $\vN^t_{[0,n]}(K)_\bO^\bI=\vN_{[0,n]}(K)_\bO^\bI$).

	As shown in \cite[Theorem 7.5]{Z-Cub2}, both the maps
	\[
		|\Ch(Y^A)| 
		=
		|\Ch(Y^A;n)| 
		\longleftarrow
		\Lhocolim_{\bc\in\Ch(Y^A)} \vN(\square^{{\vee}\bn^{\bc}})_\bO^\bI
		\longrightarrow
		\vN^t_{[0,n]}(|Y^A|)_\bO^\bI \tag{*}
	\]
	are weak homotopy equivalences that are functorial
	and, hence, $\Sigma_A$--equivariant.
	All three spaces appearing have the $\Sigma_A$--homotopy types of $\Sigma_A$--CW-complexes:
	for the hocolim it follows from the fact 
	that the spaces $\vN(\square^{{\vee}\bn^{\bc}})_\bO^\bI$ are contractible \cite[Proposition 6.2]{Z-Cub}.
	Thus, these maps are $\Sigma_A$--equivariant homotopy equivalences.
	The action of $\Sigma_A$ on all these spaces is free:
	\begin{itemize}
	\item
		For $\vN^t_{[0,n]}(Y^A)_\bO^\bI$ it follows from Proposition \ref{p:YActionFree},
	\item
		For the hocolim, from the existence of an equivariant map into a free $\Sigma_A$--space,
	\item
		It will be shown below (Propositions \ref{l:FreeActionDA} and \ref{p:ChRComp})
	that $\Sigma_A$ acts freely on $\Ch(Y^A)$
	and, therefore, on $|\Ch(Y^A)|$.
	\end{itemize}
	As a consequence of equivariant Whitehead theorem \cite[Corollary 3.3]{May-Equivariant},
	the maps in (*) are $\Sigma_A$--homotopy equivalences.
\end{proof}

\section{Doubly ordered sets}
\label{s:DOS}

Throughout the whole section, $A$ is a fixed set having $n$ elements.

\def\trcup{\mathop{\bar\cup}}

\begin{df}\label{d:StrictOrders} {\ }
\begin{itemize}
\item
	\emph{A strict partial order} on $A$ is a transitive and irreflexive relation.
\item
	A strict partial order $<$ is \emph{total} if $a<b$ or $b<a$ for every $a\neq b\in A$.
\item
	\emph{The union} of strict partial orders ${<_1}$ and $<_2$ on $A$, denoted by ${<_1}\trcup{<_2}$, is the transitive closure of the union of relations ${<_1}$ and ${<_2}$;
this is not necessarily irreflexive.
\item
	A strict partial order ${<}$ on $A$ is \emph{semi-linear} if there exists a surjective function 
	\begin{equation}\label{e:HFun}
		h=h({<}):A\to \{1,\dotsc,l({<})\}
	\end{equation}
	that induces $<$, ie, such that $a<b$ if and only if $h(a)<h(b)$.
	Such a function is unique.		
\end{itemize}
\end{df}

\begin{lem}\label{l:UnionOfSemiLinearOrders}
	Let ${<_1}$ and ${<_2}$ be semi-linear strict partial orders on $A$.
	If ${<_1}\trcup{<_2}$ is irreflexive, then it is a semi-linear strict partial order.
\end{lem}
\begin{proof}
	Assume that ${<_1}\trcup{<_2}$ is a strict partial order.
	Denote $h_1=h({<_1})$, $h_2=h(<_2)$ and let
	\[
		B=\im(h_1 \times h_2)\subseteq \{1,\dotsc,l({<_1})\}\times \{1,\dotsc,l({<_2})\}.
	\]
	There are no elements $a,a'\in A$ such that $h_1(a)<h_1(a')$ and $h_2(a')<h_2(a)$.
	As a consequence, the relation $<_B$ on $B$:
	\[
		(b_1,b_2)<_B (b_1',b_2')\iff \text{$b_1<b_1'$ or $b_2< b_2'$}
	\]
	is a total strict order on $B$. Let $p:(B,{<}_B)\to \{1<\dotsm<m\}$ be the unique order-preserving bijection; it is easy to verify that the composition $p\circ(h_1\times h_2)$ induces ${<_1}\trcup{<_2}$.
\end{proof}

\begin{df}\label{d:DoublyOrderedSet}
	{\ }
	\begin{itemize}
	\item
		\emph{A double order} on $A$ is a pair of strict partial orders ${\dle}=(\xle,\yle)$ on $A$ such that for every $a,b\in A$ at least one of the conditions $a\xle b$, $b\xle a$, $a\yle b$, $b\yle a$ holds.
	\item
		A double order $\dle$ is \emph{regular} if $\xle$ is semi-linear and $a\xle b$ implies that neither $a\yle b$ nor $b \yle a$.	
	\item
		The union of double orders ${\dle_1}=({\xle_1},{\yle_1})$ and ${\dle_2}=({\xle_2},{\yle_2})$
		is the pair of relations ${\dle_1}\mathop{\bar\cup}{\dle_2}=({\xle_1}\mathop{\bar\cup}{\xle_2},{\yle_1}\mathop{\bar\cup}{\yle_2})$,
		which is a double order if both ${\xle_1}\mathop{\bar\cup}{\xle_2}$ and ${\yle_1}\mathop{\bar\cup}{\yle_2}$ are irreflexive.
	\item
		A double order $\dle$ is \emph{semi-regular} if it is a union of regular double orders, ie, there exist regular double orders ${\dle_1},\dotsc,{\dle_k}$ such that ${\xle}=\bar\bigcup_{i=1}^k {\xle_i}$ and ${\yle}=\bar\bigcup_{i=1}^k {\yle_i}$.
	\end{itemize}
\end{df}

Let $D(A)$ (resp. $R(A)$, $R^+(A)$) denote the set of double (resp. regular, semi-regular) orders on $A$.
The permutation group $\Sigma_A$ acts on $D(A)$ (resp. $R(A)$, $R^+(A)$) from the right by ${\dle}\sigma=({\xle}\sigma, {\yle}\sigma)$,
\[
	a \mathop{{\xle}\sigma}b \iff \sigma(a)\xle\sigma(b),\qquad
	a\mathop{{\yle}\sigma}b \iff \sigma(a)\yle\sigma(b).
\]
\begin{prp}\label{l:FreeActionDA}
	$\Sigma_A$ acts freely on $D(A)$. As a consequence, $\Sigma_A$ also acts freely on $R(A)$ and $R^+(A)$.
\end{prp}
\begin{proof}
	Let ${\dle}\in D(A)$, $\sigma\in \Sigma_A$ 
	and assume that ${\dle}\sigma={\dle}$.
	We will show that $\sigma$ is the identity inductively with respect to $n=|A|$.
	If $n=1$, this is clear.
	Let $M\subseteq A$ be the set of all $\xle$--minimal elements of $A$.
	Obviously $\sigma(M)=M$.
	Since no two elements of $M$ are comparable by $\xle$,
	then $M$ is totally ordered by ${\yle}$
	and $\sigma|_M$ is the identity.
	The restriction $\sigma|_{A\setminus M}$ is the identity by the inductive hypothesis.
\end{proof}

\begin{prp}\label{p:DOAction}
	Let ${\dle}$ be a regular double order on $A$ and let  $\sigma\in\Sigma_A$.
	Then ${\dle}\trcup{{\dle}\sigma}$ is a double order if and only if $\sigma$ is the identity.
\end{prp}
\begin{proof}
	Assume that ${\dle}\trcup{{\dle}\sigma}$ is a double order.	
	Clearly $h({\xle}\sigma)=h({\xle})\circ \sigma$,
	which implies that for all $i\in \{1,\dotsc,l({\xle})\}$,
	the preimages $h({\xle})^{-1}(i)$ and $h({\xle}\sigma)^{-1}(i)$
	are equipotent.
	Assume that the functions $h({\xle}\sigma)$ and $h({\xle})$ are not equal.
	Then there exist $a,b\in A$
	such that $h({\xle})(a)<h({\xle})(b)$ and $h({\xle}\sigma)(b)<h({\xle}\sigma)(a)$,
	ie $a\xle b \mathrel{{\xle}\sigma} a$: a contradiction.
	As a consequence, $h({\xle})\sigma=h({\xle})$.
	
	In a similar way we show that for every $i\in \{1,\dotsc,l({\xle})\}$ and every $a,b\in h({\xle})^{-1}(i)=h({\xle}\sigma)^{-1}(i)$,
	$a\yle b$ if and only if $a \mathrel{{\yle}\sigma} b$.
	Eventually, $\sigma$ is the identity. The inverse implication is obvious.
\end{proof}

Let us introduce two (reflexive) partial orders on $D(A)$:
\begin{align}
	(\dle_1)\subseteq (\dle_2) &\iff (\xle_1)\subseteq (\xle_2) \land (\yle_1)\subseteq(\yle_2),\\
	(\dle_1)\sqsubseteq (\dle_2) &\iff (\xle_1)\subseteq (\xle_2) \land (\yle_1)\supseteq(\yle_2).
\end{align}

In the remaining part of this section we will show that the following pairs of $\Sigma_A$--spaces are $\Sigma_A$--homotopy equivalent:
\begin{itemize}
\item
	$|(R(A),{\sqsubseteq})|$ and $|(R^+(A),{\subseteq})|$ (Proposition \ref{p:RRComp}).
\item
	$|(R(A),{\sqsubseteq})|$ and $\georel{\Ch(Y^A)}$ (Proposition \ref{p:ChRComp}),
\end{itemize}
An equivalence between $|(R^+(A),{\subseteq})|$ and $\OConf(A,\R^2)$ will be proven in the next section  (Proposition \ref{p:RConfEq}).


For strict partial orders ${<}_1$ and ${<}_2$ on $A$ define a relation
\begin{equation}
	a({<}_1{\setminus}{<}_2)b 
	\iff
	\text{$a<_1 b$ and neither $a<_2 b$ nor $b<_2 a$.}
\end{equation}

\begin{lem}
	If ${\dle}=({\xle},{\yle})\in R^+(A)$, then $({\xle},{\yle}{\setminus}{\xle})\in R(A)$.
\end{lem}
\begin{proof}
		By Lemma \ref{l:UnionOfSemiLinearOrders}, $\xle$ is semi-linear. 
		For $a,b,c\in A$ we have
		\[\
			a \mathrel{{\yle}{\setminus}{\xle}} b \mathrel{{\yle}{\setminus}{\xle}}c 
			\implies
			(\text{$h({\xle})(a)=h({\xle})(b)=h({\xle})(c)$ and $a\yle c$})
			\implies 
			a \mathrel{{\yle}{\setminus}{\xle}}c,
		\]
		so ${\yle}{\setminus}{\xle}$ is transitive.
		Since $\yle$ is irreflexive, then so is ${{\yle}{\setminus}{\xle}}$.
		Hence, ${\yle}{\setminus}{\xle}$ is a strict partial order.
		If $a\mathrel{{\yle}{\setminus}{\xle}}b$, then obviously neither $a\xle b$ nor $b\xle a$;
		as a consequence, $({\xle},{\yle}{\setminus}{\xle})\in R(A)$.
\end{proof}

\begin{lem}\label{l:DLECompare}
	Assume that ${\dle}_1 \subseteq {\dle_2}\in D(A)$. Then 
	\begin{enumerate}[\normalfont(a)]
	\item
		${\yle_2}{\setminus}{\xle_2}={\yle_1}{\setminus}{\xle_2}\subseteq {\yle_1}\setminus{\xle_1}$,
	\item
		$({\xle_1},{\yle_1}{\setminus}{\xle_1})\sqsubseteq ({\xle_2},{\yle_2}{\setminus}{\xle_2})$.
	\end{enumerate}
\end{lem}	
\begin{proof}
	The inclusions ${\yle_1}{\setminus}{\xle_2}\subseteq {\yle_2}{\setminus}{\xle_2}$ and 
	${\yle_1}\setminus{\xle_2}\subseteq {\yle_1}\setminus{\xle_1}$ are obvious.
	If $a \mathrel{{\yle_2}{\setminus}{\xle_2}} b$, then neither $a\xle_2 b$ nor $b\xle_2 a$, which implies that neither $a\xle_1 b$ nor $b\xle_1 a$.
	As a consequence, $a$ and $b$ must be comparable by $\yle_1$. 
	Since $a\yle_2 b$, it follows that $a\yle_1 b$ and then $a \mathrel{{\yle_1}{\setminus}{\xle_2}} b$.
	This proves (a), from which (b) follows immediately.
\end{proof}

Lemma \ref{l:DLECompare}.(b) implies that the map
\begin{equation}
	F:(R^+(A),{\subseteq})\ni ({\xle},{\yle})\mapsto ({\xle},{\yle}{\setminus}{\xle}) \in (R(A),{\sqsubseteq})
\end{equation}
preserves order (ie, it is a functor).
In order to define an ``inverse" functor we need to pass to barycentric subdivisions.
\emph{The barycentric subdivision} of a poset $(P,{\leq})$ is the poset $(\sd(P),{\subseteq})$ whose elements are chains in $P$.
Note that the functor 
	\begin{equation}\label{e:MaxEq}
		\max: \sd(P)\ni (p_0\lneq\dotsm\lneq p_n) \mapsto p_n \in (P,{\leq})
	\end{equation}
induces a homotopy equivalence $\georel{\max}: \georel{\sd(P)}\to \georel{P}$.

The formula
\begin{equation}
	G((\dle_0)\sqsubsetneq (\dle_1)\sqsubsetneq \dotsm\sqsubsetneq (\dle_r))
	=
	\bigcup_{i=0}^r {\dle_i}
	=
	\left(\xle_r, \yle_0\right).
\end{equation}
defines a functor $G:\sd(R(A),\sqsubseteq)\to (R^+(A),\subseteq)$.
Indeed, if
\[
	\big((\dle_0)\sqsubsetneq (\dle_1)\sqsubsetneq \dotsm\sqsubsetneq (\dle_r)\big)
	\subseteq
	\big((\dle'_0)\sqsubsetneq (\dle'_1)\sqsubsetneq \dotsm\sqsubsetneq (\dle'_s)\big),
\] 
then
\[
	G\big((\dle_0)\sqsubsetneq (\dle_1)\sqsubsetneq \dotsm\sqsubsetneq (\dle_r)\big)
	=
	\bigcup_{i=0}^r {\dle_i}
	\subseteq
	\bigcup_{i=0}^s {\dle'_i}
	=
	G\big((\dle'_0)\sqsubsetneq (\dle'_1)\sqsubsetneq \dotsm\sqsubsetneq (\dle'_s)\big).
\]
Let us consider the diagram of posets:
\begin{equation}
	\begin{diagram}
		\node{(R(A),{\sqsubseteq})}
		\node{(R^+(A),{\subseteq})}
			\arrow{w,t}{F}
	\\
		\node{\sd(R(A),{\sqsubseteq})}
			\arrow{ne,t}{G}
			\arrow{n,l}{\max}
		\node{\sd(R^+(A),{\subseteq})}
			\arrow{w,t}{\sd(F)}
			\arrow{n,r}{\max}
	\end{diagram}
\end{equation}

Using Lemma \ref{l:DLECompare}.(a) for $({\xle_r},{\yle_r})\subseteq ({\xle_r},{\yle_0})$, we obtain
\begin{multline*}
	FG((\dle_0)\sqsubsetneq (\dle_1)\sqsubsetneq \dotsm\sqsubsetneq (\dle_r))
	= 
	F\left({\xle_r}, {\yle_0}\right)
	=
	\left({\xle_r}, {\yle_0\setminus\xle_r}\right)
	=
	\left({\xle_r}, {\yle_r\setminus\xle_r}\right)	
	\\
	=
	\left({\xle_r},{\yle_r}\right)
	=
	\max((\dle_0)\sqsubsetneq (\dle_1)\sqsubsetneq \dotsm\sqsubsetneq (\dle_r)).
\end{multline*}
Thus, the left upper triangle of the diagram commutes.
Furthermore, 
\begin{multline*}
	G\circ \sd(F) ((\dle_0)\subsetneq (\dle_1)\subsetneq \dotsm\subsetneq (\dle_r))
	=
	G\left(\left(\xle_0, (\yle_0{\setminus}\xle_0)\right) \sqsubsetneq\dots,\sqsubsetneq \left(\xle_r, (\yle_r{\setminus} \xle_r)\right) \right)\\
	=
	\left(\xle_r, (\yle_0{\setminus} \xle_0)\right)
	\subseteq 
	\left(\xle_r, \yle_0\right)
	\subseteq 
	\left(\xle_r, \yle_r\right)
	=
	\max((\dle_0)\subsetneq (\dle_1)\subsetneq \dotsm\subsetneq (\dle_r)).	
\end{multline*}
Hence, the right bottom triangle does not necessarily commute. Nevertheless, there is a natural transformation of functors $G\circ \sd(F)\Rightarrow \max:\sd(R^+(A),{\subseteq})\to (R^+(A),{\subseteq})$, so the geometric realizations of functors $G\circ\sd(F)$ and $\max$ are homotopic.

All the maps above are $\Sigma_A$--equivariant, so we have proven the following.

\begin{prp}\label{p:RRComp}
	The maps $\georel{F}: \georel{(R^+(A),{\subseteq})} \to \georel{(R(A),{\sqsubseteq})}$ and 
	\[
		\georel{(R(A),{\sqsubseteq})}
		\cong
		\georel{\sd(R(A),{\sqsubseteq})}
		\xrightarrow{\georel{G}}
		\georel{(R^+(A),{\subseteq})}
	\]
	are mutual $\Sigma_A$--equivariant homotopy inverses.\qed
\end{prp}

Now we will show that $(R(A),\sqsupseteq)$ is $\Sigma_A$--homotopy equivalent to $\Ch(Y^A)$
Recall that $Y^A$ is non-self-linked (Proposition \ref{p:YANonSelfLinked}) and, therefore, $\Ch(Y^A)$ is a poset (Lemma \ref{l:ChNSL}.(a)). We will write $\bc\leq \be$ whenever $\Ch(Y^A)(\bc,\be)\neq\emptyset$.

\begin{lem}
	Let $\bc=((c_1,{<_1}),\dotsc,(c_l,<_l))\in \Ch(Y^A)$.
	For every $a\in A$ there is a unique integer $h_\bc(a)\in\{1,\dotsc,l\}$ such that $c_{h_\bc(a)}(a)=*$.
\end{lem}
\begin{proof}
	For all $a\in A$, we have $d^0(c_1)(a)=0$, which implies $c_1(a)\in\{0,*\}$ and, similarly, $c_l(a)\in\{*,1\}$. The condition $d^1(c_i)=d^0(c_{i+1})$ implies that 
	\[
		(c_i(a),c_{i+1}(a))\in\{(0,0), (0,*), (*,1), (1,1)\}.
	\]
	Thus, there exists a unique index $i$ with $c_i(a)=*$.
\end{proof}

For a cube chain $\bc=((c_1,{<_1}),\dots,(c_l,{<_l}))\in \Ch(Y^A)$ define a regular double order ${\dle}_\bc=(\xle_\bc,\yle_\bc)$ on $A$ by
\begin{itemize}
\item
	$a\xle_\bc b$ iff $h_\bc(a)<h_\bc(b)$ (ie, $h({\xle}_\bc)=h_\bc$),
\item
	$a\yle_\bc b$ iff $h_\bc(a)=h_\bc(b)=j$ and $a <_{j} b$ (ie, ${\yle}_\bc=\bigcup_j {<_j}$).
\end{itemize}
For a regular double order ${\dle}=(\xle,\yle)\in R(A)$,
define functions $c^{\dle}_j:A\to\{0,*,1\}$ ($j\in\{1,\dotsc,l({\xle})\}$)
\[
	c^{\dle}_j(a)=
	\begin{cases}
		1 & \text{if $h({\xle})(a)<j$}\\
		* & \text{if $h({\xle})(a)=j$}\\
		0 & \text{if $h({\xle})(a)>j$}
	\end{cases}
\]
and let ${<_j}$ be the restriction of $\yle$ to $(c^{\dle}_j)^{-1}(*)=h({\xle})^{-1}(j)$. Finally, let
\[
	\bc^{\dle}=\big((c^{\dle}_1,{<_1}),\dotsc,(c^{\dle}_{l},{<_{l}})\big)\in\Ch(Y^A),
\] 
where $l=l({\xle})$.

\begin{prp}\label{p:ChRComp}
	The maps
	\[
		\Ch(Y^A)\ni \bc \mapsto {\dle}_\bc \in (R(A),\sqsupseteq)
	\]
	and
	\[
		(R(A),\sqsupseteq)\ni {\dle} \mapsto \bc^{\dle} \in \Ch(Y^A)
	\]
	are mutually inverse $\Sigma_A$--equivariant isomorphisms of posets.
\end{prp}
\begin{proof}
It is elementary to check that these maps are mutually inverse $\Sigma_A$--equivariant bijections.

Fix cube chains $\bc=((c_1,{<_1}),\dotsc,(c_l,{<_l})),\be=((e_1,{<'_1}),\dotsc,(e_m,{<'_m}))\in\Ch(Y^A)$ and let $\dle_\bc,\dle_\be\in R(A)$ be the corresponding double orders.
Assume that $(\dle_\bc)\sqsupseteq (\dle_\be)$. Since $({\xle_\bc})\supseteq ({\xle_\be})$, then there exists an increasing surjective function $j$ such that the diagram
	\[
		\begin{diagram}
			\node{A}
				\arrow{e,t}{h_\bc}
				\arrow{se,b}{h_\be}
			\node{\{1,\dotsc,l({<_\bc})\}}
				\arrow{s,r}{j}
		\\
			\node{}
			\node{\{1,\dotsc,l({<_\be})\}}
		\end{diagram}
	\]
	commutes.
For every $i\in \{1,\dotsc,l({<_\bc})\}$ we have
\[
	c_i(a)=*
	\iff 
	h_{\bc}(a)=i
	\implies
	h_{\be}(a)=j(i)
	\iff
	e_{j(i)}(a)=*,
\]
and for $\mu\in\{0,1\}$,
\begin{align*}
	e_{j(i)}(a)=0
	\iff
	h_{\be}(a)<j(i)
	\implies
	h_{\bc}(a)<i
	\iff
	c_i(a)=0,\\
	e_{j(i)}(a)=1
	\iff
	h_{\be}(a)>j(i)
	\implies
	h_{\bc}(a)>i
	\iff
	c_i(a)=1.
\end{align*}
Moreover, for $a,b\in A$ such that $c_i(a)=c_i(b)=*$ we have (since ${\yle_\bc}\subseteq{\yle_\be}$)
\[
	a<_i b
	\iff
	a\yle_{\bc} b
	\implies
	a\yle_\be b
	\iff
	a<'_{j(i)}b.
\]
As a consequence, $(c_i,<_i)$ is a face of $(e_{j(i)},<'_{j(i)})$ (Lemma \ref{l:YAFace}).
From Lemma \ref{l:ChNSL} follows that $\bc\leq \be$.

An argument that $\bc\leq \be$ implies ${{\dle}_{\bc}}\sqsupseteq {{\dle}_{\be}}$ is similar and will be omitted.
\end{proof}

\section{Nerve lemma}

Let $G$ be a finite group and 
let $X$ be a right free $G$--space.
We assume that $X$ is paracompact and has the $G$--homotopy type of a $G$--CW-complex.
Let $\scU=\{U_i\}_{i\in I}$ be a finite open cover of $X$.

\begin{df}
	A cover $\scU$ is
	\begin{enumerate}[\normalfont(a)]
	\item
		\emph{$G$--equivariant} if there exists an action of $G$ on $I$ such that $U_ig=U_{ig}$,
	\item
		\emph{proper $G$--equivariant} if additionally $U_i\cap U_{ig}=\emptyset$ for every $i\in I$ and every $1\neq g\in G$,
	\item
		\emph{good} if every non-empty intersection of sets in $\scU$ is contractible,
	\item
		\emph{complete} if $\scU$ is closed with respect to non-empty intersections.
	\end{enumerate}
\end{df}

\emph{The nerve} $N\scU$ of $\scU$ is the family of subsets $J\subseteq I$ such that $\bigcap_{i\in J} U_i$ is non-empty.
This is a poset with respect to inclusion.
Thus, $\scU$ can be regarded as a functor
\[
	\scU:(N\scU,\supseteq)\ni J \mapsto \bigcap_{i\in J} U_i \in \Top.
\]

The classical Nerve Lemma states that if $\scU$ is good, then both maps
\begin{equation}\label{e:GNerveSeq}
	|(N\scU,\supseteq)| \overset{p}\longleftarrow \hocolim_{(N\scU,\supseteq)} \scU \overset{q}\longrightarrow \colim_{(N\scU,\supseteq)} \scU\cong X
\end{equation}
are homotopy equivalences.
Here $p$ is induced by the unique transformation of $\scU$ into the trivial functor on $(N\scU,\supseteq)$,
and $q$ is the natural projection from the homotopy colimit into the colimit.

Assume that $\scU$ is good and proper $G$--equivariant.
The action of $G$ on the indexing set $I$ 
induces a $G$--action on $N\scU$.
The properness of $\scU$ implies that this action is free.
As a consequence, $G$ acts freely on $|(N\scU,\supseteq)|$.

There is a presentation of the homotopy colimit as a quotient
\[
	\hocolim_{(N\scU,\supseteq)} \scU
	\simeq 
	\coprod_{n\geq 0}
	\coprod_{\substack{J_0\supseteq \dotsm\supseteq J_n\\J_i\in N\scU}}
	\Big(\Delta^n \times \bigcap_{i\in J_0} U_i\Big)
	/\sim.
\]
The induced action of $G$ on $\hocolim_{(N\scU,\supseteq)} \scU$
is given by the formula
\[
	(J_0\supseteq \dotsm\supseteq J_n,(t_0,\dotsc,t_n),x)g
	=
	(J_0g\supseteq \dotsm\supseteq J_ng,(t_0,\dotsc,t_n),xg).
\]
Clearly both $p$ and $q$ are $G$--equivariant maps;
as a consequence, $\hocolim_{(N\scU,\supseteq)} \scU$ is a free $G$-space.
Thus, all the spaces in (\ref{e:GNerveSeq}) are free and have the homotopy types of $G$--CW-complexes.
By the Nerve Lemma, both $p$ and $q$ are homotopy equivalences,
and, by the equivariant Whitehead theorem,
$G$--homotopy equivalences.
We have proven the following.

\begin{lem}
	Assume that $G$ is a finite group,
	$X$ a paracompact free $G$--space having the homotopy type of a $G$--CW-complex,
	and $\scU$ is a good proper $G$--equivariant cover of $X$.
	Then the spaces $|(N\scU,\supseteq)|$ and $X$ are $G$--homotopy equivalent.\qed
\end{lem}

If $\scU$ is additionally complete,
then $X$ is $G$--homotopy equivalent to the nerve of an even smaller category.
The family $\scU$ is a free $G$--poset with respect to inclusion.
There is a pair of order-preserving maps
\begin{align*}
	\Phi:(N\scU,\supseteq) \ni \{U_{i_0},\dotsm,U_{i_r}\} &\mapsto \bigcap_{s=0}^r U_{i_s} \in (\scU,\subseteq)\\
	\Psi: (\scU,\subseteq) \ni U_i & \mapsto \{U_j\in\scU\st U_j\supseteq U_i \}\in (N\scU,\supseteq).
\end{align*}
Here $\Phi$ is well-defined thanks to the completeness of $\scU$.
The composition $\Phi\Psi$ is the identity while for every $\{U_{i_0},\dotsm,U_{i_r}\}\in N\scU$ we have 
\[
	\{U_{i_0},\dotsm,U_{i_r}\} \subseteq \Psi\Phi(\{U_{i_0},\dotsm,U_{i_r}\}),
\]
which defines a natural transformation $\Id \Rightarrow \Psi\Phi$ of endofunctors on $(N\scU,\subseteq)$.

As a consequence, $\Phi$ and $\Psi$ induce homotopy equivalences between the geometric realizations of these posets.
Moreover both $\Phi$ and $\Psi$ are $G$--equivariant. As a consequence, we obtain the following.

\begin{lem}\label{l:NerveComplete}
	Let $G$ be a finite group,
	$X$ a paracompact $G$--space having the $G$--homotopy type of a G-CW-complex,
	 and $\scU$ a good complete proper $G$--equivariant cover of $X$.
	Then $X$ is $G$--homotopy equivalent to $|(\scU,\subseteq)|$.\qed
\end{lem}

Now we will apply Lemma \ref{l:NerveComplete} to prove that $|(R^+(A),{\sqsubseteq})|$ and $\OConf(A,\R^2)$ are $\Sigma_A$--homotopy equivalent.
Recall that
\[
	\OConf(A,\R^2)=\{f=(f_x,f_y):A\to \R^2\st \text{$f$ is injective}\},
\]
and $\Sigma_A$ acts freely on $\OConf(A,\R^2)$ from the right by precomposition.

For a double order ${\dle}=({\xle},{\yle})\in D(A)$ define a subset 
\begin{equation}
	U({\dle})=\{(f_x,f_y):A\to \R^2\st \forall_{a,b\in A}\;  \text{$a\xle b\implies  f_x(a)<f_x(b)$ and $a\yle b \implies f_y(a)<f_y(b)$}  \}
\end{equation}
of $\map(A,\R^2)\simeq \R^{2n}$.

\begin{prp}\label{p:UProps}
	For every ${\dle}\in D(A)$, the set $U({\dle})$ is a non-empty, open and convex subset of $\OConf(A,\R^2)$.
\end{prp}
\begin{proof}
	Let $f=(f_x,f_y)\in U({\dle})$.
	For every $a\neq b\in A$ at least one of the conditions $a\xle b$, $b\xle a$, $a\yle b$, $b\yle a$ holds.
	As a consequence, $f_x(a)\neq f_x(b)$ or $f_y(a)\neq f_y(b)$, and then, $U({\dle})\subseteq \OConf(A,\R^2)$,
	
	Extend $\xle$ and $\yle$ to strict total orders ${\xle'}$ and ${\yle'}$, respectively.
	There are unique order-preserving bijections  $h_x:(A,{\xle'})\to (\{1,\dotsc,n\},{<})$, $h_y:(A,{\yle'})\to (\{1,\dotsc,n\},{<})$.
	Then $h=(h_x,h_y)\in U({\dle'})\subseteq U({\dle})$, so that $U({\dle})$ is non-empty.
	Finally, $U({\dle})$ is open and convex since it is an intersection of open half-spaces.
\end{proof}

The following is obvious.
\begin{prp}\label{p:CoverInters}
	Let ${\dle_1},\dotsc,{\dle_k}$ be double orders on $A$ and let ${\dle}=\bar\bigcup_{i=1}^k {\dle_i}$. Then
	\[
		\bigcap_{i=1}^k U(\dle_i)=
		\begin{cases}
			U(\dle) & \text{if $\dle$ is a double order,}\\
			\emptyset & \text{otherwise.}
		\end{cases}
		\qed
	\]
\end{prp}

\begin{prp}\label{p:CovProp}
	The family $\scU_A=\{U(\dle)\}_{{\dle}\in R^+(A)}$ is a proper $\Sigma_A$--equivariant good complete cover of $\OConf(A,\R^2)$.
\end{prp}
\begin{proof}
	For arbitrary $f=(f_x,f_y)\in \OConf(A,\R^2)$ define partial orders $\xle_f$ and $\yle_f$ by
	\begin{itemize}
	\item	
		$a\xle_f b$ iff $f_x(a)<f_x(b)$,
	\item
		$a \yle_f b$ iff $f_x(a)=f_x(b)$ and $f_y(a)<f_y(b)$.
	\end{itemize}
	Clearly ${\dle}_f=(\xle_f, \yle_f)$ is a regular double order
	and $f\in U_{{\dle}_f}$;
	therefore, $\{U({\dle})\}_{{\dle}\in R(A)}\subseteq \{U({\dle})\}_{{\dle}\in R^+(A)}$
	is already a cover of $\OConf(A,\R^2)$.
	Propositions \ref{p:CoverInters} and \ref{p:UProps} imply that
	$\scU_A$ is complete and good, respectively.

	Clearly $U({\dle})\sigma=U({\dle}\sigma)$ for all ${\dle}\in D(A)$,
	and permutations of $A$ preserve semi-regularity of double orders,
	so $\{U(\dle)\}_{{\dle}\in R^+(A)}$ is $\Sigma_A$--equivariant.
		
	For every ${\dle}\in R^+(A)$ there exists ${\dle'}\in R(A)$ such that ${\dle'}\subseteq {\dle}$.
	Hence, for $1\neq\sigma\in\Sigma_A$,
	${\dle'} \mathop{\bar\cup} {\dle'}\sigma$	is not a double order (Proposition \ref{p:DOAction});
	therefore
	\[
		U({\dle})\cap U({\dle}\sigma)
		\subseteq 
		U({\dle'})\cap U({\dle'}\sigma)
		=
		\emptyset
	\]
	by Proposition \ref{p:CoverInters}.
\end{proof}

Note that $U({\dle_1})\subseteq U({\dle_2})$ if and only if ${\dle_1}\supseteq{\dle_2}$.
Immediately from Proposition \ref{p:CovProp} and Lemma \ref{l:NerveComplete} we obtain.

\begin{prp}\label{p:RConfEq}
	The $\Sigma_A$--spaces $|(R^+(A),{\supseteq})|$ and $\OConf(A,\R^2)$ are $\Sigma_A$--homotopy equivalent.\qed
\end{prp}

\section{Applications}

In this section we present several application of the results obtained above.

\subsection*{Invariants of directed paths on $\square$--sets}

\newcommand\PiInv[3]{\mathrm{IB}_{#2}(#1,#3)}
\newcommand\PiInvS[3]{\mathrm{I\Sigma}_{#2}(#1,#3)}

Every bi-pointed $\square$--set $K$ admits a unique bi-pointed $\square$--map $K\to Z$, which in turn induces a map
\[
	\vP(K)_\bO^\bI \to \vP(Z)_\bO^\bI.
\]
Thus, for every $n\geq 0$ we have a map
\[
	R_{n}(K):
	\vP(K;n)
	\cong
	\vP(\tilde{K}_n)_\bO^\bI
	\to
	\vP(\tilde{Z}_n)_\bO^\bI
	\xrightarrow{\simeq}
	\UConf(n,\R^2),
\]
which is functorial with respect to bi-pointed $\square$--maps.
This allows to define some invariants of bi-pointed $\square$--sets, or rather their execution spaces.
\begin{enumerate}
\item
	For every $\alpha\in\vP(K;n)$, there is a representation (ie, a conjugacy class of homomorphism of groups)
	\[
		\PiInv{K}{n}{\alpha}:\pi_1(\vP(K;n), \alpha)\to \pi_1(\UConf(n,\R^2))\simeq B_n,
	\]
	in the braid group on $n$ threads.
	The composition of $\PiInv{K}{n}{\alpha}$ with the canonical surjection $B_n\to \Sigma_n$ will be denoted by $\PiInvS{K}{n}{\alpha}$.
\item
	Every element $\omega\in H^k(B_n)\simeq H^k(\UConf(n,\R^2))$ defines a ``characteristic class", namely, the cohomology class $R_n(K)^*(\omega)\in H^k(\vP(K;n))$.
\end{enumerate}

It turns out that these invariants are trivial if $K$ is a sculpture \cite{Pratt} or a Euclidean complex \cite{RZ} (as shown in \cite{Z-Cub2}, these classes are equivalent).
\begin{prp}
	If $K\subseteq \square^n$, then the map $R_n(K)$ is null-homotopic.
	In particular, both $\PiInv{K}{n}{\alpha}$ and $R_n(K)^*(\omega)$ are trivial, for all $\alpha\in \vP(K;n)_\bO^\bI$ and all $\omega\in H^k(\UConf(n,\R^2))$.
\end{prp}
\begin{proof}
	$R_n(K)$ factors through $\vP(\square^n)_\bO^\bI$, which is a contractible space.
\end{proof}

\obsol{
\begin{cor}
	Every d-path in both $\vP(Y^A)_\bO^\bI$ and $\vP(\tilde{Z}_n)_\bO^\bI$ is tame.
\end{cor}
}

\subsection*{Spaces of directed paths on $\square$--sets}
Here we prove that every execution space of a finite bi-pointed $\square$--set has the homotopy type of a CW-complex, which extends the result of Raussen \cite{Raussen-Trace}.
Furthermore, every such space can be realized, up to finite covering, as the execution space of a finite non-self-linked $\square$--set.

\begin{prp}\label{p:Pullbacks}
	Let
	\[
		\begin{diagram}
			\node{K\times_M L}
				\arrow{s}
				\arrow{e}
			\node{L}
				\arrow{s,r}{q}
		\\
			\node{K}
				\arrow{e,t}{p}
			\node{M}
		\end{diagram}
	\]
	be a pullback diagram of $\square$--sets. If $K$ and $L$ are finite, then
	\[
		\begin{diagram}
			\node{|K\times_M L|}
				\arrow{s}
				\arrow{e}
			\node{|L|}
				\arrow{s,r}{q}
		\\
			\node{|K|}
				\arrow{e,t}{p}
			\node{|M|}
		\end{diagram}
	\]	
	is a pullback diagram.
\end{prp}
\begin{proof}
	The geometric realizations of projections $K\times_M L\to K$ and $K\times_M L\to L$
	induce the map $f:|K\times_M L|\to |K|\times_{|M|}|L|$
	such that
	\[
		f([(c,c');x_1,\dotsc,x_n])=([c;x_1,\dotsc,x_n], [c';x_1,\dotsc,x_n])
	\]
	for $c\in K[n]$, $c'\in L[n]$ and $p(c)=q(c')$.
	Every point $x\in |K|\times|L|$ has a canonical presentation $([c;x_1,\dotsc,x_k],[c';y_1,\dotsc,y_l])$ and $x\in |K|\times_{|M|}|L|$ if and only if $k=l$, $(x_i)=(y_i)$ and $p(c)=q(c')$. Thus, $f$ is a bijection and then a homeomorphism, since $|K\times_M L|$ is compact. 
\end{proof}

\begin{prp}\label{p:Pullbacks2}
	The functor $\vP(-)_\bO^\bI:\dTop_*^* \to \Top$ preserve pullbacks.
\end{prp}
\begin{proof}
		We have $\vP(-)_\bO^\bI=\dTop_*^*(\vI,-)$.
\end{proof}

\begin{lem}\label{l:ZnProd}
	The map $\vP(K\times \tilde{Z}_n)_\bO^\bI\to \vP(K;n)_\bO^\bI$ induced by the projection on the first factor is a homeomorphism.
\end{lem}
\begin{proof}
	From Propositions \ref{p:Pullbacks} and \ref{p:Pullbacks2} follows that 
\begin{equation}
		\begin{diagram}
			\node{\vP(K\times \tilde{Z}_n)_\bO^\bI}
				\arrow{s}
				\arrow{e}
			\node{\vP(\tilde{Z}_n)_\bO^\bI}
				\arrow{s}
		\\
			\node{\vP(K\times Z)_\bO^\bI}
				\arrow{e}
			\node{\vP(Z)_\bO^\bI}
		\end{diagram}
\end{equation}
	is a pullback diagram. 
	Furthermore, the right-hand vertical map is the inclusion
	\[
		\vP(\tilde{Z}_n)_\bO^\bI\xrightarrow{\cong} \vP(Z;n)_\bO^\bI \subseteq \vP(Z)_\bO^\bI.
	\]
	Since the length of paths is preserved by $\square$--maps, we have
	\[
		\vP(K\times \tilde{Z}_n)_\bO^\bI\simeq \vP(K\times Z;n)_\bO^\bI\cong \vP(K;n)_\bO^\bI.\qedhere
	\]
\end{proof}

\begin{prp}\label{p:PKisCW}
	If $K$ is a bi-pointed $\square$--set, then $\vP(K)_\bO^\bI$ has the homotopy type of a CW-complex.
\end{prp}
\begin{proof}
Let $A$ be a set having $n$ elements. The diagram
\begin{equation}
		\begin{diagram}
			\node{\vP(K\times Y^A)_\bO^\bI}
				\arrow{s}
				\arrow{e}
			\node{\vP(Y^A)_\bO^\bI}
				\arrow{s,r}{\pi}
		\\
			\node{\vP(K\times \tilde{Z}_n)_\bO^\bI}
				\arrow{e}
			\node{\vP(\tilde{Z}_n)_\bO^\bI}
		\end{diagram}
\end{equation}
is a pullback diagram. 
The right-hand column is a $\Sigma_A$--principal bundle (Proposition \ref{p:MainCov}),
so the left-hand column is also a $\Sigma_A$--principal bundle.
Since $A$ is non-self-linked (Proposition \ref{p:YANonSelfLinked}),
$K\times Y^A$ is non-self-linked also.
Therefore, $\vP(K\times Y^A)_\bO^\bI$ has the homotopy type of a CW-complex \cite{Raussen-Trace}
and even the $\Sigma_A$--homotopy type of a $\Sigma_A$--CW-complex (see Proof of \ref{p:ZHtpCW}).
Then $\vP(K\times \tilde{Z}_n)_\bO^\bI=\vP(K\times Y^A)_\bO^\bI/\Sigma_A$ also has the homotopy type of a CW-complex.
Since $\vP(K\times\tilde{Z}_n)_\bO^\bI\simeq \vP(K;n)_\bO^\bI$, 
the conclusion follows by applying the presentation (\ref{e:LengthDecomp}).
\end{proof}

As an immediate consequence of Proposition \ref{p:PKisCW} and Theorem \ref{thm:CubeChainModel}, we obtain:
\begin{cor}\label{c:ChHE}
	For every bi-pointed $\square$--set $K$, the spaces $\vP(K)_\bO^\bI$ and $\georel{\Ch(K)}$ are naturally homotopy equivalent.
\end{cor}

The next statement shows that, when calculating homotopy types of execution spaces, one can restrict to non-self-linked $\square$--sets,
at least up to finite covering.
For a bi-pointed d-space $X$ and $\alpha\in \vP(X)_\bO^\bI$ let $\vP(X,\alpha)$ denote the connected component of $\vP(X)_\bO^\bI$ containing $\alpha$.

\begin{prp}
	Let $K$ be a bi-pointed $\square$--set and let $\alpha\in \vP(K;n)_\bO^\bI$.
	\begin{enumerate}[\normalfont(a)]
	\item
		There exists a non-self-linked bi-pointed $\square$--set $K'$, a $\square$--map $f:K'\to K$ and $\alpha'\in \vP(K',n)_\bO^\bI$
		such that $|f|\circ \alpha'=\alpha$
		and the map $\vP(f):\vP(K',\alpha')_\bO^\bI\to \vP(K,\alpha)_\bO^\bI$
		is a finite covering.
	\item
		If $\PiInvS{K}{n}{\alpha}$ is trivial, then we can require that $\vP(f)$ is a homeomorphism.
	\end{enumerate}
\end{prp}
\begin{proof}
	Fix a set $A$ having $n$ elements.
	Let $f:K'=K\times Y^A\to K$ be the projection on the first factor.
	Since $Y^A$ is non-self-linked, so is $K'$.
	The map induced by $f$ is the composition
	\[
		\vP(K')_\bO^\bI=\vP(K\times Y^A)_\bO^\bI \to \vP(K\times \tilde{Z}_n)_\bO^\bI \xrightarrow{\cong}  \vP(K\times Z;n)_\bO^\bI \xrightarrow{\cong}  \vP(K;n)_\bO^\bI.
	\]
	 Hence, it is a principal $\Sigma_A$--bundle and (a) is proven.
	 
	 If $\PiInvS{K}{n}{\alpha}$ is the trivial homomorphism,
	 then $\vP(K\times Y^A)_\bO^\bI \to \vP(K\times \tilde{Z}_n)_\bO^\bI$ is a product bundle,
	 which shows (b).
\end{proof}

\subsection*{James construction}
Let $(X,x_0)$ be a pointed topological space and let $J(X)=J(X,x_0)$ be the free topological monoid spanned by $(X,x_0)$.
Elements of $J(X)$ are finite sequences of elements of $X$,
and two sequences are equivalent if one can be obtained from the other by adding or removing $x_0$.
Namely,
\begin{equation}
	J(X)=\colim\left(J_0(X)\subseteq J_1(X)\subseteq J_2(X)\subseteq \dotsm \right),
\end{equation}
where $J_n(X)=X^n/(x_1,\dotsc,x_i,x_0,x_{i+1},\dotsc,x_n)\sim(x_1,\dotsc,x_0,x_i,x_{i+1},\dotsc,x_n)$, and the inclusions add $x_0$ at the last position.
By the classical result of James \cite{James}, the natural map $J(X)\to \Omega\Sigma X$ is a weak homotopy equivalence.

A sequence of maps 
\[
	f_n:I^n\ni (x_1,\dotsc x_n)\mapsto (e^{2\pi i x_1},\dotsc,e^{2\pi i x_n}) \in  J_n(S^1,1)
\]
satisfies relations $f_n\circ \delta^\varepsilon_i = f_{n-1}$ for all $n$, $\varepsilon$ and $i$ (see (\ref{e:DeltaMaps})).
Thus, these maps glue to the maps $f^{(n)}:|Z^{(n)}|\to J_n(S^1)$,
which are bijective.
Here $Z^{(n)}$ stands for the $n$--skeleton of $Z$.
Indeed, every point of $|Z^{(n)}|$ (respectively $J_n(S^1)$) has a unique presentation as a sequence of at most $n$ elements of $(0,1)$ (resp. $S^1\setminus\{1\}$).
Since the spaces $|Z^{(n)}|$ are compact, $f^{(n)}$'s are homeomorphisms. 
By passing to the colimit we obtain the following.

\begin{prp}
	The map $f=\colim(f^{(n)}): |Z| \to J(S^1)$ is a homeomorphism.
\end{prp}

As a consequence, the spaces $|Z|$ and $\Omega S^2$ are homotopy equivalent.
The initial and the final vertex of $Z$ are the same.
Thus, we have maps
\[
	\vP(Z)_\bO^\bI\subseteq  \Omega|Z| \cong \Omega J(S^1) \xrightarrow{\simeq} \Omega^2(S^2)=\map_*(S^2,S^2)
\]
and 
\[
	\UConf(n,\R^2) \simeq \vP(Z;n)_\bO^\bI \subseteq \vP(Z)_\bO^\bI \to \Omega^2 S^2.
\]
This suggests that unordered configurations spaces on $\R^2$ can be regarded as ``directed" self-maps of $S^2$,
although an interpretation of this fact is not clear for the authors.

\subsection*{Presentation of configuration spaces as nerves of certain categories}

\def\RDoSet{\mathbf{RDoSet}}

Let $\cC$ be a finite category and let $G$ be a finite group.
Assume that $G$ acts freely on $\cC$ from the right, ie. $cg\neq c$ for every $1\neq g\in G$ and every object $c\in \cC$.
\emph{The quotient category} $\cC/G$ is the category with objects $\Ob(\cC/G)=\Ob(\cC)/G$ and morphisms
\[
	(\cC/G)(cG,c'G)
	=
	\left( \coprod_{g\in G} \coprod_{g'\in G}\cC(cg, c'g')\right)/G
	\simeq
	\coprod_{g\in G} \cC(c,c'g)
	\simeq
	\coprod_{g\in G} \cC(cg,c').	
\]
The composition of morphisms $\alpha G\in (\cC/G)(cG,c'G)$ and $\beta G\in (\cC/G)(c'G,c''G)$
is $(\beta'\alpha')G$,
where $\alpha'\in \coprod_{g\in G} \cC(cg,c')$  (resp. $\beta'\in\coprod_{g\in G} \cC(c',c''g)$) is the unique morphism such that $\alpha'G=\alpha G$ (resp. $\beta'G=\beta G$). The projection $\cC\to\cC/G$ is a functor.

\begin{rem}
	The quotient category $\cC/G$ can be defined for arbitrary, non necessarily free action of $G$, as a colimit of the suitable diagram in the category of small categories. But, in general, there is no guarantee that morphisms in the quotient category $\cC/G$ are exactly orbits of morphisms in $\cC$.
\end{rem}

\begin{lem}\label{l:GNerve}
	If a group $G$ acts freely on a category $\cC$, then the map $|\cC|/G \to |\cC/G|$ induced by the projection $\cC\to\cC/G$ is a homeomorphism.
\end{lem}
\begin{proof}
	It is easy to check that $\cC\to\cC/G$ induces an isomorphism of simplicial sets $N(\cC)/G\to N(\cC/G)$.
	The conclusion follows.
\end{proof}

From the main theorem and Lemma \ref{l:GNerve} follows that there is a homotopy equivalence
\[
	\UConf(n;\R^2)=\OConf(A,\R^2)/\Sigma_A \simeq |(R(A),\sqsupseteq)|/\Sigma_A \overset{\simeq}{\longrightarrow} |(R(A),\sqsupseteq)/\Sigma_A|,
\]
which is a presentation of $\UConf(n;\R^2)$ as a nerve of a category.
Following \cite{Paliga}, we present below an explicit description of the category $(R(A),\sqsupseteq)/\Sigma_A$.
We will use notation introduced in Section \ref{s:DOS} without further notice.

\begin{df}
Let $\cE_n$ be the category whose objects are subsets of $\{1,\dotsc,n-1\}$.
For
\[
		B=\{b_1<\dotsm<b_{r-1}\}\supseteq B'=\{b'_1<\dotsm<b'_{r'-1}\}\subseteq \{1,\dotsc,n-1\},
\]
morphisms from $B$ to $B'$
are permutations $\varphi$ of $\{1,\dotsc,n\}$
satisfyng the conditions:
\begin{itemize}
\item[(x)] 
	$\varphi(\{b'_{p-1}+1,\dotsc,b'_{p}\})=\{b'_{p-1}+1,\dotsc,b'_{p}\}$ for $p\in\{1,\dotsc,r'\}$ (we assume $b'_0=0$, $b'_{r'}=n$).
\item[(y)]
	$\varphi(b_{p-1}+1)<\varphi(b_{p-1}+2)<\dotsm<\varphi(b_{p})$ for $p\in \{1,\dotsc,r\}$ (again, $b_0=0$, $b_{r}=n$).
\end{itemize}
For $B\not\supseteq B'$ we assume $\cE_n(B,B')=\emptyset$.
The composition in $\cE_n$ is the composition of permutations.
\end{df}

Note that (x) is equivalent to the condition
\begin{itemize}
\item[(x')]
	$\forall i,j\in\{1,\dotsc,n\}:\; (\exists b'\in B':\; i\leq b'<j) \implies \varphi(i)<\varphi(j)$,
\end{itemize}
and that $\varphi$ satisfies (x) if and only if $\varphi^{-1}$ satisfies (x).

We will construct a functor $F:(R(A),\sqsubseteq)\to \cE_n$.
For a regular double order ${\dle}\in R(A)$
let
\[
	F({\dle})=\{ b_1,\dots,b_{r-1}\}\subseteq \{1,\dotsc,n-1\}, 
\]
where $b_p=|\{a\in A\st h({\xle})(a)\leq p\}|$, $r=l({\xle})$.

\begin{df}
	Let ${\dle}\in R(A)$, $F({\dle})=\{b_1<\dotsm<b_{r-1}\}$.
	\emph{The monotonic numbering} for ${\dle}\in R(A)$ is a presentation $A=\{a_1,\dots,a_n\}$ such that
\[
	a_i\xle a_j \iff \exists p:\; i\leq b_p<j,\qquad a_i\yle a_j \iff \exists p:\; b_{p-1}<i<j\leq b_p.
\]
	Such a numbering exists and is unique.
\end{df}

For ${\dle} \sqsupseteq {\dle'}\in R(A)$ let $A=\{a_1,\dotsc,a_n\}$ and $A=\{a'_1,\dotsc,a'_n\}$ be the monotonic numberings for ${\dle}$ and ${\dle'}$, respectively.
We define $F({\dle} \sqsupseteq {\dle'})$ to be the unique permutation $\varphi\in\Sigma_n$ such that $a'_{\varphi(i)}=a_i$ for all $i$.
It is easy to verify that $F$ preserves composition.
It is a functor thanks to the following.

\begin{prp}
	If ${\dle}\sqsupseteq {\dle'}\in R(A)$, then $\varphi:=F({\dle}\sqsupseteq {\dle'})\in \cE_n(F({\dle}),F({\dle'}))$.
\end{prp}
\begin{proof}
	We have ${\xle}\supseteq {\xle'}$ and ${\yle}\subseteq{\yle'}$.
	Let $a_i$, $a'_i$, $b_p$, $b'_{p}$ be as above.

	For $i,j\in\{1,\dotsc,n\}$ we have
\[
	\exists p:\; i\leq b'_p < j
	 \iff
	 a'_i\xle' a'_j 
	 \implies
 	 a'_i\xle a'_j 
 	 \iff
 	 a_{\varphi^{-1}(i)}\xle a_{\varphi^{-1}(j)} 
 	 \implies 
 	 \varphi^{-1}(i)<\varphi^{-1}(j) 
\]	
	Thus, $\varphi^{-1}$ satisfies (x') and, therefore, $\varphi$ satisfies (x).
	Also,
	\[
		b_{p-1}<i<j\leq b_p
		\implies 
		a_i\yle a_j
		\implies
		a_i\yle' a_j
		\iff
		a'_{\varphi(i)}\yle' a'_{\varphi(j)}
		\implies
		\varphi(i)<\varphi(j)
	\]	
	and then $\varphi$ satisfies (y).
\end{proof}

\begin{prp}\label{p:FSigmaInv}
	Let $\sigma\in\Sigma_A$. Then
	\begin{enumerate}[\normalfont (a)]
	\item
		$F({\dle}\sigma)=F({\dle})$ for ${\dle}\in\ R(A)$,
	\item
		$F({\dle}\sqsupseteq{\dle'})=F({\dle}\sigma\sqsupseteq{\dle'}\sigma)$ for ${\dle}\sqsupseteq{\dle'}\in R(A)$.
	\end{enumerate}
\end{prp}
\begin{proof}
	We have $h({\xle}\sigma)=h({\xle})\sigma$, which implies (a).
	Assume that $(a_i)_{i=1}^n$ and $(a'_i)_{i=1}^n$ are monotonic numberings of ${\dle}$ and ${\dle'}$, respectively,
	and that $\varphi=F({\dle}\sqsupseteq{\dle'})$.
	Then $(\sigma^{-1}(a_i)_{i=1}^n)$ and $(\sigma^{-1}(a'_i)_{i=1}^n)$ are monotonic numberings of ${\dle}\sigma$ and ${\dle'}\sigma$, respectively.
	We have $\sigma^{-1}(a_i)=\sigma^{-1}(a'_{\varphi(i)})$
	and, therefore, $F({\dle}\sigma\sqsupseteq{\dle'}\sigma)=\varphi$.
\end{proof}

As a consequence of Proposition \ref{p:FSigmaInv}, $F$ induces the quotient functor
\[
	\bar{F}:(R(A),\sqsupseteq)/\Sigma_A \to \cE_n.
\]

\begin{prp}\label{p:BarFIso}
	The functor $\bar{F}$ is an isomorphism of categories.
\end{prp}
\begin{proof}
	Let $B=\{b_1<\dotsm<b_{r-1}\}\in\cE_n$.
	Every numbering $A=\{a_1,\dotsc,a_n\}$ is the monotonic numbering of the unique double order ${\dle}\in R(A)$,
	and all double orders ${\dle}$ such that $F({\dle})=B$ can be obtained this way.
	As a consequence, $\bar{F}$ is a bijection on objects.
	
	Now let $\varphi\in \cE_n(B,B')$, $B'=\{b'_1<\dotsm<b'_{r'}\}$, ${\dle}\in F^{-1}(B)$ and let $A=\{a_1,\dotsc,a_n\}$ be the monotonic numbering for $\dle$.
	The double order $\dle'$ given by
	\[
		a_i\xle' a_j\iff \exists p:\; \varphi(i)\leq b'_p<\varphi(j),
		\qquad
		a_i\yle' a_j\iff \exists p:\; b'_{p-1}<\varphi(i)<\varphi(j)\leq b'_p
	\]
	is the only possible regular double order such that $F({\dle})=B'$ and $F({\dle}\sqsupseteq{\dle'})=\varphi$.
	It remains 	to check that ${\dle}\sqsupseteq {\dle'}$. We have
	\[
		a_i\xle' a_j
		\iff
		\exists p: \; \varphi(i)\leq b'_p<\varphi(j)
		\overset{\text{(x)}}\implies
		\exists p:\; i\leq b'_p<j
		\overset{B\supseteq B'}\implies
		\exists q:\; i\leq b_q<j
		\iff
		a_i\xle a_j.
	\]	
	and
	\[
		a_i\yle a_j
		\iff
		\exists q:\; b_{q-1}<i<j\leq b_q
		\overset{\text{(x),(y)}}\implies
		\exists p:\; b'_{p-1}<\varphi(i)<\varphi(j) \leq b'_p
		\iff
		a_i\yle' a_j.
	\]
	Thus, $F^{-1}(\varphi)$ contains exactly one element for every $\dle\in F^{-1}(B)$, ie, it is a single orbit.
	Therefore, $\bar{F}$ is a bijection on morphisms.
\end{proof}

Combining Propositions \ref{p:RRComp}, \ref{p:RConfEq}, \ref{p:BarFIso} and Lemma \ref{l:GNerve} we obtain the following.

\begin{prp}
	For every $n\geq 0$,
	$\UConf(n,\R^2)$ is homotopy equivalent to the nerve of the category $\cE_n$.
\end{prp}


\begin{thebibliography}{FGHR}
\bibitem{FGHMR}
	L. Fajstrup, E. Goubault, E. Haucourt, S. Mimram, M. Raussen,
	\emph{Directed Algebraic Topology and Concurrency},
	Springer (2016).
\bibitem{FGR}
	L. Fajstrup, E. Goubault, M. Raussen,
	\emph{Algebraic Topology and Concurrency},
	Theor.\ Comput.\ Sci.\ \textbf{357} (2006), 241--278.
\bibitem{Grandis}
	M. Grandis,
	\emph{Directed homotopy theory, I. The fundamental category},
	Cahiers Top. Geom. Diff. Categ. \textbf{44} (2003), 281--316.
\bibitem{James}
	I. M. James,
	\emph{Reduced product spaces},
	Ann. of Math. \textbf{62} (1955), 170--197.
\bibitem{Kozlov}
	D.~Kozlov,
	\emph{Combinatorial Algebraic Topology},
	Algorithms and Computation in Mathematics \textbf{21}, Springer (2008).
\bibitem{May-Equivariant}
	J.~P.~May,
	\emph{Equivariant Homotopy and Cohomology Theory},
	CBMS Regional Conference Series in Mathematics,
	Volume 91, 1996.
\bibitem{Paliga}
	J. Paliga,
	\emph{A homotopy model for the unordered configuration space of points in the Euclidean plane},
	Master thesis, University of Warsaw, 2020.
\bibitem{Pratt}
	V. Pratt,
	\emph{Modelling concurrency with geometry},
	 Proc. of the 18th ACM Symposium on Principles of Programming Languages (1991), 311--322.
\bibitem{Raussen-Sim2}
	M. Raussen,
	\emph{Simplicial models for trace spaces II: General Higher
          Dimensional Automata}, Algebr. Geom. Topol. \textbf{12}
        (2012), 1545--1565.
\bibitem{Raussen-Trace}
	M.~Raussen,
	\emph{Trace spaces in a pre-cubical complex},
		Topology Appl. 156 \textbf{9} (2009), 1718-1728.\\
		DOI 10.1016/j.topol.2009.02.003
\bibitem{RZ}
	M. Raussen, K. Ziemia\'nski,
	\emph{Homology of spaces of directed paths on Euclidean cubical complexes},
	J. Homotopy Relat. Struct. \textbf{9} (2014) 67--84.
	DOI 10.1007/s40062-013-0045-4
\bibitem{Waner-Milnor}
	S.~Waner,
	\emph{Equivariant Homotopy Theory and Milnor's Theorem},
	Trans.~Amer.~Math.~Soc.~\textbf{258} (1980), 351--368.
\bibitem{Z-Cub}
	K.~ Ziemia\'nski
	\emph{Spaces of directed paths on pre-cubical sets}
	Appl. Algebra Eng. Commun. Comput. \textbf{28} (2017), 497--525.
	DOI: 10.1007/s00200-017-0316-0.
\bibitem{Z-Cub2}
	K. Ziemia\'nski,
	\emph{Spaces of directed paths on pre-cubical sets II}
	J Appl. and Comput. Topology \textbf{4}, 45--78 (2020).\\
	DOI 10.1007/s41468-019-00040-z
\end{thebibliography}
\end{document}